\documentclass[a4paper,11pt]{article}

\usepackage{amsmath, amssymb, amsthm, amsfonts}

\usepackage{graphics} %\graphicspath{{figures/}}

\usepackage{psfrag, booktabs}
\usepackage{color, epsfig, float}
\usepackage[active]{srcltx}
\usepackage[bookmarksopen, colorlinks, linkcolor = blue, urlcolor = red,
            citecolor = red, menucolor = blue]{hyperref}

\usepackage{float}
\floatstyle{plain}
\newfloat{algorithm}{thp}{alg}
\floatname{algorithm}{Algorithm}

\newtheorem{theorem}{Theorem}[section]
\newtheorem{lemma}[theorem]{Lemma}

\newtheorem{propo}[theorem]{Proposition}

\newtheorem{remark}[theorem]{Remark}
\newtheorem{assump}[theorem]{Assumption}

\newcommand\la{\lambda}

\newcommand{\ipl}{\langle}
\newcommand{\ipr}{\rangle}

\newcommand\summ{\textstyle\sum\limits}
\newcommand\T{\textstyle}
\newcommand\D{\displaystyle}

\newcommand\norm[1]{\|#1\|}
\newcommand\R{\mathbb{R}}

\newcommand\N{\mathbb{N}}
\DeclareMathOperator{\argmin}{arg\, min}

\begin{document}
% ========================================================================

\title{On inertial iterated Tikhonov methods for solving ill-posed problems}
\setcounter{footnote}{1}

\author{
J.C.~Rabelo%
\thanks{Department of Mathematics, Federal Univ.\,of Piaui,
        64049-550 Teresina, Brazil}
%       \href{mailto:joelrabelo@ufpi.edu.br}{\tt joelrabelo@ufpi.edu.br}.}
\and
A.~Leit\~ao%
\thanks{\mbox{EMAp, Getulio Vargas Fundation, Praia de Botafogo 190,
        22250-900 Rio de Janeiro, Brazil}}
\thanks{\mbox{On leave from Department of Mathematics, Federal Univ.\,of
        St.\,Catarina, 88040-900 Floripa, Brazil}} $^\P$
\and
A.\,L.\,Madureira%
\thanks{LNCC, Av.\,Get\'ulio Vargas 333, P.O.\,Box\,95113,
        25651-070 Petr\'opolis, Brazil}
%       \href{mailto:alm@lncc.br}{\tt alm@lncc.br}.}
}

\date{\small \today}

\maketitle

\vspace{-0.8cm}
\begin{abstract}
In this manuscript we propose and analyze an implicit two-point type method
(or inertial method) for obtaining stable approximate solutions to linear
ill-posed operator equations. The method is based on the iterated
Tikhonov (iT) scheme.
We establish convergence for exact data, and stability and semi-convergence
for noisy data. Regarding numerical experiments we consider: i) a 2D Inverse
Potential Problem, ii) an Image Deblurring Problem; the computational efficiency
of the method is compared with standard implementations of the iT method.
\end{abstract}

{\let\thefootnote\relax\footnotetext{Emails:
\href{mailto:joelrabelo@ufpi.edu.br}{\tt joelrabelo@ufpi.edu.br}, \
\href{mailto:acgleitao@gmail.com}{\tt acgleitao@gmail.com}, \
\href{mailto:alm@lncc.br}{\tt alm@lncc.br}.}}

\medskip

\bigskip

\noindent
\begin{minipage}[t]{0.6\textwidth}
% \raggedright
{\noindent\small This manuscript is dedicated to Professor Johann
Baumeister (Oberzeitldorn, Bavaria,\,7.\,March\,1944) on the occasion
of his 80th birthday and in recognition of his valuable contributions
to Inverse Problems, Control\,Theory, Numerical Analysis and Variational
Calculus.

He is the author of one of the first books on Inverse Problems
\cite{Bau87}, published by Vieweg \& Sohn in 1987, which influenced a
whole generation of mathematicians.\,A zealous adviser, his door has
always been open to his students,
welcoming them with attention and patience.

Johann Baumeister is currently Emeritus Professor at the Department of
Mathematics of the Johann-Wolfgang von Goethe Universit\"at, Frankfurt
am Main, Germany.}
\end{minipage}
\hfill\noindent
\begin{minipage}[t]{0.32\textwidth}
% \raggedleft
\vskip-0.2cm
\centerline{\includegraphics[width=\textwidth]{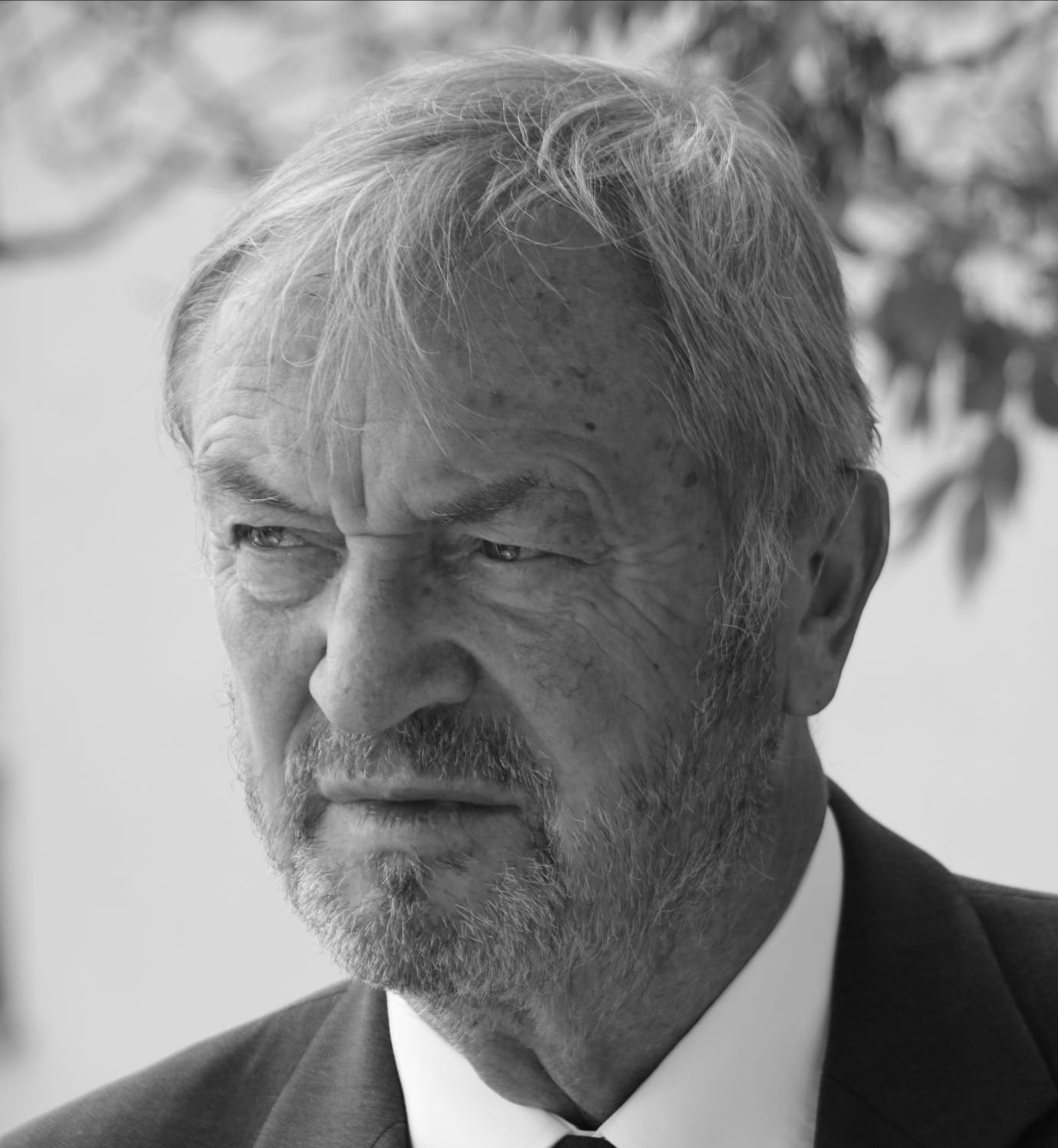}}
\end{minipage}

% ------------------------------------------------------------------------
\section{Introduction} \label{sec:intro}

\subsection*{Problems of interest}

In a typical \emph{inverse problem} setting~\cite{Bau87,Kir96,Rie03},
let $X$ and $Y$ be Hilbert spaces, and consider the problem of determining
an unknown quantity $x \in X$ from given data $y \in Y$, i.e. an unknown
quantity of interest $x$ (which cannot be directly measured) has to be
identified, based on information obtained from some set of measured data $y$.

A relevant point is that, in practice, the exact data $y \in Y$ is unavailable.
Instead, only approximate measured data $y^\delta \in Y$ satisfying
\begin{equation}\label{eq:noisy-i}
    \norm{ y^\delta - y } \ \le \ \delta \, ,
\end{equation}
is accessible. Here, the known $\delta > 0$ represents the level of
noise (or uncertainty in the measurements). The available noisy data
$y^\delta \in Y$ are obtained by indirect measurements of $x \in X$;
a process represented by the model
\begin{equation}\label{eq:inv-probl}
    A\,x \ = \ y^\delta ,
\end{equation}
where $A: X \to Y$, is a bounded linear ill-posed operator, whose inverse
$A^{-1} : Y \to X$ either does not exist, or is not continuous.

\subsection*{Standard iterations}

We recall two families of iterative methods for obtaining
stable approximate solutions to the linear ill-posed operator equation
\eqref{eq:inv-probl}, namely (I) the family of explicit iterative methods
defined by
$$
x_{k+1}^\delta \ = \ x_k^\delta - \gamma_k A^*(A x_k^\delta - y^\delta)
\, ,\ k = 0, 1, \dots \, ;
$$
(II) the family of implicit iterative methods
$$
x_{k+1}^\delta \ = \ x_k^\delta - \la_k A^*(A x_{k+1}^\delta - y^\delta)
\, ,\ k = 0, 1, \dots
$$
Here $A^*: Y \to X$ is the adjoint operator to $A$. The iterations start
at a given initial guess $x_0 \in X$ (which may incorporate {\em a priori}
information about the exact solution of $Ax = y$).

Both families of methods, defined by choices of $\gamma_k$ and $\la_k$,
can be interpreted as iterative schemes for solving the normal equation
$A^*A x = A^* y^\delta$, which is the optimally condition for the least
square problem $\min_x \norm{Ax - y^\delta}^2$ \cite{Bau87}.

In the explicit schemes, step computations are not expensive and the
parameters $\gamma_k$ play the role of step size control.
Appropriate choices of $\gamma_k$ lead to different methods, e.g., the Landweber
iteration \cite{Lan51}, the steepest descent method \cite{Sch96}, the minimal
error method \cite{EngHanNeu96} or generalizations of these methods \cite{LS18}.
Explicit iterative type methods are extensively discussed in the inverse problems
literature. Their regularization properties are well established for both linear
and nonlinear operator equations \cite{Bau87, EngHanNeu96, Kir96, KalNeuSch08}.
It is well known in particular that the above mentioned Landweber, steepest
descent and minimal error methods present slow convergence rates
\cite{EngHanNeu96, KalNeuSch08}.

On the other hand, each step of the implicit methods (also known as iterated
Tikhonov (iT) methods \cite{HG98, BLS20} or proximal point (PP) methods
\cite{Mar70, Roc76}) corresponds to computing  $x_{k+1}^\delta :=
\argmin_x \big\{ \la_k \| Ax - y^\delta \|^2 + \| x - x_k^\delta \|^2 \big\}$.
The parameters $\la_k$ are appropriately chosen Lagrange multipliers \cite{BLS20}.
Here, the computation of the iterative step is more demanding than in the
explicit type methods, since $x_{k+1}^\delta$ is obtained by solving the
linear system $(\la_k A^*A + I) x = x_k^\delta + \la_k A^* y^\delta$
at each step. However, the iT type methods require less iterations
than the explicit methods, to compute an approximate solution of similar
quality \cite{EngHanNeu96, KalNeuSch08}.

The literature on iT type methods is extensive. Various aspects are
investigated, including regularization properties \cite{Eng87, GS00, KN08,
KalNeuSch08}, convergence rates \cite{HG98, Sch93a}, {\em a posteriori}
strategies for choosing the Lagrange multipliers \cite{BLS20}, and a
cyclic version of the iT method \cite{CBL11}.

\subsection*{Two-point iterations}

Two-point iteration schemes can be interpreted as a generalization of the
above explicit/im\-plicit methods.
In 1983 Nesterov discussed in \cite{Nes83} a strategy to accelerate the
convergence of explicit methods (I), where the computation of $x_{k+1}^\delta$ depends on
the last two iterates $x_k^\delta$ and $x_{k-1}^\delta$. In the notation of
\eqref{eq:inv-probl} the Nesterov accelerated forward-backward scheme reads
\begin{equation} \label{eq:nesterov}
x_{k+1}^\delta \ = \ w_k^\delta - \gamma A^*(A w_k^\delta - y^\delta) \, ,
\quad {\rm where} \quad
w_k^\delta \ = \ x_k^\delta + \ \alpha_k (x_k^\delta - x_{k-1}^\delta) \, ,
\ k \geq 0 \, .
\end{equation}
Here $x_0^\delta \in X$ is an initial guess, $x_ {-1}^\delta = x_0^\delta$,
$\alpha_k = (k-1)/(k-1+\alpha)$ with $\alpha \geq 3$, and $\gamma$ is a
scaling parameter.
The calculation of $x_k$ and $w_k$ is explicit, and as expensive to compute
as the step of the (explicit) Landweber method.
This explicit two-point iteration scheme improves the theoretical rate
of convergence for the functional values $\norm{A x_k - y^\delta}^2$
from the standard ${\cal O}(1/k)$ to ${\cal O}(1/k^2)$ \cite{Nes83};
in 2016, Attouch and Peypouquet \cite{AP16} where able to improve this rate
to $o(1/k^2)$ for $\alpha > 3$.

Explicit two-point schemes are successfully considered in the context
of {\em Iterative Shrinkage-Thresholding Algorithms} (ISTAs); among the
resulting explicit two-point methods we mention the {\em Fast ISTA}
(or FISTA), considered in 2009 by Beck and Teboulle \cite{BeTe09},
and the {\em Two-Step ISTA} (or TwIST) proposed in 2007 by Bioucas-Dias
and Figueiredo \cite{BF07}. These methods are designed for solving linear
inverse problems arising in image processing and image restoration
(a large amount of inertial schemes for optimization and inverse problems
is available in the literature; we refer to the monograph of Chambolle
and Pock \cite{CaPo16}, the book of Beck \cite{Be17} and the references
therein).

More recently, the Nesterov accelerated scheme was considered as
an alternative for obtaining stable solutions to ill-posed operator
equations.
In 2017, Neubauer considered the Nesterov scheme with a stopping rule
(the discrepancy principle) and proved convergence rates under standard
source conditions \cite{Neu17}.
In 2021 Kindermann revisited the Nesterov scheme for linear ill-posed
problems \cite{Ki21} and proved that this explicit two-point method is
an optimal-order iterative regularization method.
In 2017 Hubmer and Ramlau \cite{HR17} considered the Nesterov scheme
(with discrepancy principle) for nonlinear ill-posed problems (under
the Scherzer condition \cite[Eq.~(11.6)]{EngHanNeu96}). Convergence for
exact data is proven as well as semi-convergence in the noisy data case.
The same authors considered in \cite{HR18} the Nesterov scheme for
nonlinear ill-posed problems with a locally convex residual functional.
\medskip

The implicit two-point iterative schemes are known in the optimization
literature under the name of {\em inertial proximal methods} and have
been well analyzed by many authors over the last two decades (see,
e.g., \cite{Al03, AP16, OCBP14, MO03}).
To the best of our knowledge, implicit two-point type methods have
not yet been considered in the inverse problems literature, and this
manuscript represents an attempt to do so.

\subsection*{The inertial iterated Tikhonov method}

In this article we propose and analyze an implicit two-point iteration
scheme that can be interpreted as a generalization of the iT method;
our iteration relates to the inertial method proposed in 2001 by Alvarez
and Attouch \cite{AA01}.
We propose this implicit two-point method as a viable alternative for
computing stable approximate solutions to the ill-posed operator
equation \eqref{eq:inv-probl}, and investigate its numerical efficiency.

The method under consideration consists in choosing appropriate
non-negative sequences $(\alpha_k)$, $(\la_k)$, and defining at
each iterative step the extrapolation
$w_k^\delta := x_k^\delta + \alpha_k (x_k^\delta - x_{k-1}^\delta)$;
the next iterate $x_{k+1}$ is defined by
\begin{equation} \label{eq:iit-step}
x_{k+1}^\delta \ := \
\argmin_x \big\{ \la_k \| Ax - y^\delta \|^2 + \| x - w_k^\delta \|^2 \big\}
\, , \ k = 0, 1, \dots
\end{equation}
where $x_{-1} = x_0\in X$ are given. For obvious reasons we refer to this
implicit two-point method as {\em inertial iterated Tikhonov} (iniT) method.
\medskip

The outline of the manuscript is as follows:
In Section~\ref{sec:iteration} we introduce and analyze the inertial
iterated Tikhonov (iniT) method, proving convergence for exact data in
Section~\ref{ssec:2.2}, and presenting stability and semi-convergence
results in Section~\ref{ssec:2.3}.
In Section~\ref{sec:numerics} the iniT method is tested in a two-dimensional
Inverse Potential Problem and an Image Deblurring Problem.
Section~\ref{sec:conclusions} is devoted to final remarks an conclusions.
% In Appendix~\ref{sec:app}

% ------------------------------------------------------------------------
\section{The iteration} \label{sec:iteration}

In this section we introduce and analyze the inertial iterated Tikhonov
(iniT) method considered in this manuscript. 
In Section~\ref{ssec:2.1} the iniT method is presented and some basic
inequalities are established. Convergence for exact data is proven in
Section~\ref{ssec:2.2}. Stability and semi-convergence results are proven
in Section~\ref{ssec:2.3}.

\subsection{Description of method} \label{ssec:2.1}

We begin by addressing the exact data case $y^\delta = y$. Define the
quadratic (square residual) functional $f: X \ni x \to f(x) :=
\frac12 \norm{Ax-y}^2 \in \R^+$.

Denoting the current iterate by $x_k \in X$, for $k \geq 0$, the step
of the proposed iniT method consists in two parts: first compute the
extrapolation point $w_k \in X$,
\begin{subequations} \label{def:iniT}
\begin{equation} \label{def:iniT-w}
w_k \ := \ x_k + \alpha_k \, (x_k - x_{k-1}) \, ;
\end{equation}
in the sequel, the next iterate $x_{k+1} \in X$ is defined as the
solution of
\begin{equation} \label{def:iniT-x}
\la_k \nabla f(x_{k+1}) + x_{k+1} - w_k \ = \ 0 \, .
\end{equation}
\end{subequations}
Notice that \eqref{def:iniT-x} is equivalent to
$\la_k \nabla f(x_{k+1}) + (x_{k+1}-x_k) - \alpha_k (x_k-x_{k-1}) = 0$,
i.e. \eqref{def:iniT-x} cor\-responds to an inertial proximal point update
(compare with the iterative step of the inertial proximal method in
\cite[Equation~(${\cal A}_1$)]{AA01}).

Here $x_0 \in X$ plays the role of an initial guess and $x_{-1} := x_0$.
Moreover, $(\alpha_k) \in [0,\alpha)$ for some $\alpha \in (0,1)$, and
$(\la_k) \in \R^+$ are given sequences.
Notice that, if $\alpha_k \equiv 0$ then \eqref{def:iniT-x} corresponds to the
standard iT iteration for exact data, i.e. $x_{k+1}$ is defined as the solution
of $\la_k \nabla f(x_{k+1}) + x_{k+1} - x_k = 0$.

It is straightforward to see that \eqref{def:iniT-x} is equivalent to $G_k \, x_{k+1}
= \la_k A^* y + w_k$, where $G_k := (\la_k A^* A + I): X \to X$ is a positive
definite operator with spectrum contained in the interval $[1, 1 + \la_k\norm{A}^2]$.
Consequently, the iterate $x_{k+1}$ is uniquely defined by \eqref{def:iniT-x}.

In what follows we present the iniT method in algorithmic form

\begin{algorithm}[h!]
\begin{center}
\fbox{\parbox{13.5cm}{
$[0]$ choose an initial guess $x_0 \in X$; \ set $x_{-1} := x_0$; \ $k := 0$;
\medskip

$[1]$ choose $\alpha \in [0,1)$ \ and \ $(\la_k)_{k\geq0} \in \R^+$;
\medskip

$[2]$ {\bf while \ $\norm{A x_k - y} > 0$ \ do} 
\smallskip

\ \ \ \ \ $[2.1]$ choose $\alpha_k \in [0,\alpha]$;
\smallskip

\ \ \ \ \ $[2.2]$ $w_k := x_k + \alpha_k (x_k - x_{k-1})$;
\smallskip

\ \ \ \ \ $[2.3]$ compute $x_{k+1} \in X$ as the solution of
\smallskip

\centerline{$\la_k \nabla f(x_{k+1}) + x_{k+1} - w_k \ = \ 0$;}
\smallskip

\ \ \ \ \ $[2.4]$ $k := k+1$;
\smallskip

\ \ \ \ \ {\bf end while}

% $[2]$ {\bf for \ $k \geq 0$ \ do} 
% \smallskip
% 
% \quad \ $[2.1]$ {\bf if \ $\norm{A x_k - y} > 0$ \ then}
% 
% \quad \ \ \ \ \ \ \ \ \ \ \ \ choose $\alpha_k \in [0,\alpha]$;
% 
% \quad \ \ \ \ \ \ \ \ \ \ \ \ $w_k := x_k+\alpha_k(x_k-x_{k-1})$;
% 
% \quad \ \ \ \ \ \ \ \ {\bf else}
% 
% \quad \ \ \ \ \ \ \ \ \ \ \ \ $\alpha_k = 0$;
% 
% \quad \ \ \ \ \ \ \ \ \ \ \ \ $w_k := x_k$;
% 
% \quad \ \ \ \ \ \ \ \ {\bf end if}
% \smallskip
% 
% \quad \ $[2.2]$ compute $x_{k+1} \in X$ as the solution of
% 
% \centerline{$\la_k \nabla f(x_{k+1}) + x_{k+1} - w_k \ = \ 0$;}
% 
% \quad \ {\bf end for}

\smallskip
} }
\end{center} \vskip-0.5cm
\caption{Inertial iterated Tikhonov method (iniT) in the exact data case.}
\label{alg:init-exact}
\end{algorithm}

% \begin{remark} \label{rem:station}
% In Algorithm~\ref{alg:init-exact}, if $A x_{k_0} = y$ for some $k_0 \in \N$,
% the extrapolation step \eqref{def:iniT-w} is not applied (i.e. $\alpha_{k_0} = 0$
% and  $w_{k_0} = x_{k_0}$). Thus, it follows from Step~[2.2] that $x_{k_0+1}
% = x_{k_0}$. Consequently, the iteration becomes stationary with $x_k = x_{k_0}$
% and $\alpha_k=0$, for $k \geq k_0$.
% \end{remark}

\begin{remark} \label{rem:station}
In Algorithm~\ref{alg:init-exact}, if $A x_{k_0} = y$ for some $k_0 \in \N$,
the iteration stops after computing $x_{-1}, x_0, \dots, x_{k_0}$ \ and \
$w_0, \dots, w_{k_0-1}$.
Notice that Algorithm~\ref{alg:init-exact} generates (infinite) sequences
$(x_k)_{k\in\N}$ and $(w_k)_{k\in\N}$ if and only if $A x_k \not= y$, for
all $k \in \N$.
\end{remark}

% The remaining of this subsection is devoted to establishing preliminary
% results and estimates (see Lemma~\ref{lemma:aux} and Proposition~%
% \ref{prop:aux}) related to sequences $(x_k)$, $(w_k)$ generated by
% Algorithm~\ref{alg:init-exact}.

The remaining of this subsection is devoted to investigating preliminary
properties of the sequences $(x_k)$, $(w_k)$ generated by Algorithm~%
\ref{alg:init-exact} (see Lemma~\ref{lemma:aux} and Proposition~\ref{prop:aux}).

\begin{lemma} \label{lemma:aux}
Let $(x_k)$, $(w_k)$ be sequences generated by Algorithm~\ref{alg:init-exact}.
% with $(\alpha_k)$ defined as in Step~[2.1].
Given $x \in X$, it holds
\begin{equation} \label{eq:w}
\norm{w_k - x}^2 \, = \, (1+\alpha_k) \norm{x_k - x}^2 - \alpha_k
\norm{x_{k-1} - x}^2 + \alpha_k (1+\alpha_k) \norm{x_k - x_{k-1}}^2 ,
\ k = 0, 1, \dots
\end{equation}
\end{lemma}
\begin{proof}
From \eqref{def:iniT-w} follows
\begin{equation} \label{eq:xk}
x_k \ = \ (1+\alpha_k)^{-1} w_k + \alpha_k (1+\alpha_k)^{-1} x_{k-1}
\, , \ k = 0, 1, \dots
\end{equation}
Consequently,
\begin{equation} \label{eq:wk}
w_k - x_{k-1} \ = \ (1 + \alpha_k) (x_k - x_{k-1}) \, , \ k = 0, 1, \dots
\end{equation}
Arguing with \eqref{eq:xk}, together with the identity $(1+\alpha_k)^{-1} + \alpha_k
(1+\alpha_k)^{-1} = 1$ and the strong convexity of the functional $\norm{\cdot}^2$,%
\footnote{For $0 < p < 1$ and $z$, $w \in X$ it holds $\norm{pz + (1-p)w}^2
= p\norm{z}^2 + (1-p)\norm{w}^2 - p(1-p)\norm{z-w}^2$.}
we obtain
\begin{eqnarray*}
\norm{x_k - x}^2
& = & \norm{(1+\alpha_k)^{-1} (w_k-x) + \alpha_k (1+\alpha_k)^{-1} (x_{k-1}-x)}^2
      \\
& = & (1+\alpha_k)^{-1} \norm{w_k - x}^2 + \alpha_k (1+\alpha_k)^{-1}
      \norm{x_{k-1} - x}^2
      \\
&   & - \, \alpha_k (1+\alpha_k)^{-2} \norm{w_k - x_{k-1}}^2 \, ,
\end{eqnarray*}
for $k \geq 0$.
To complete the proof, it is enough to substitute \eqref{eq:wk} in the
last identity.
% The lemma follows now substituting \eqref{eq:wk} in this last identity.
\end{proof}

\begin{assump} \label{ass:alpha}
Given $\alpha\in[0,1)$ and a convergent series $\sum_k \theta_k$ of
nonnegative terms, let 
$$
\alpha_0 = \alpha \quad {\rm and} \quad
\alpha_{k} :=
        \left\{ \begin{array}{cl}
          \min\left\{\dfrac{\theta_k}
                {\|x_k-x_{k-1}\|^2} , \theta_k , \alpha \right\}
                & , \ {\rm if } \ \norm{x_k-x_{k-1}} > 0 \\[2ex]
          0     & , \ {\rm otherwise}
        \end{array} \right.
\, , \ k \geq 1 \, . 
$$
\end{assump}

\begin{remark} \label{rem:alpha}
Assumption~\ref{ass:alpha} implies the summability of the series
$\sum_{k\geq0} \alpha_k \| x_k - x_{k-1} \|^2$.

To prove the next proposition, an auxiliary result is needed (see
Appendix~A). In order to apply this result, the summability of the
above mentioned series is required.
\end{remark}

\begin{propo} \label{prop:aux}
Let $(x_k)$, $(w_k)$ be sequences generated by Algorithm~\ref{alg:init-exact},
with $(\la_k)$, $(\alpha_k)$ defined as in Steps~[1] and [2.1] respectively.
The following assertions hold true
\medskip

\noindent
(a) If $x^*\in X$ is a solution of $Ax = y$ then
$$
\norm{w_k - x^*}^2 - \norm{x_{k+1} - x^*}^2 =
\norm{\la_k A^*(A x_{k+1} - y)}^2 + 2 \la_k \norm{Ax_{k+1}-y}^2
\, , \ k = 0, 1, \dots
$$

\noindent
Additionally, if $(\alpha_k)$ satisfies Assumption~\ref{ass:alpha}, it holds
\medskip

\noindent
(b) The sequences $(x_k)$ and $(w_k)$ are bounded.
\medskip

\noindent
(c) The series
\begin{equation} \label{eq:series}
\summ_{k=0}^{\infty} \la_k \norm{Ax_{k+1} - y}^2 , \ \
\summ_{k=0}^{\infty} \norm{\la_k A^*(Ax_{k+1}-y)}^2 , \ \
\summ_{k=0}^{\infty} \norm{x_{k+1} - w_k}^2 , \ \
\summ_{k=0}^{\infty} \norm{x_{k+1} - x_k}^2
\end{equation}
are summable.
\end{propo}
\begin{proof}
From \eqref{def:iniT-x} follows
\begin{eqnarray*}
\| w_k - x \|^2 - \| x_{k+1} - x\|^2
& = & \| w_k - x_{k+1} \|^2
  + 2 \ipl w_k - x_{k+1} , x_{k+1} - x \ipr
\\
& = & \| \la_kA^*(Ax_{k+1}-y) \|^2
  + 2 \la_k \ipl A x_{k+1} - y , A x_{k+1} - Ax \ipr \, ,
\end{eqnarray*}
for $x \in X$ and $k \geq 0$.
Assertion~(a) is a direct consequence of this equation with $x = x^*$.
\medskip

\noindent {\bf Assertion (b):}
Define $\varphi_k := \norm{x_k-x^*}^2$, for $k \geq -1$, \ and \
$\eta_k := \alpha_k \norm{x_k - x_{k-1}}^2$, for $k \geq 0$.
It follows from Assertion~(a) and \eqref{eq:w} with $x = x^*$ that
\begin{multline} \label{eq:gamma}
\varphi_{k+1} - \varphi_k + \norm{\la_k A^*(A x_{k+1} - y)}^2
    + 2 \la_k \norm{A x_{k+1} - y}^2 \ = \\
\alpha_k (\varphi_k - \varphi_{k-1})
    + \alpha_k (1+\alpha_k) \norm{x_k - x_{k-1}}^2 , \ k \geq 0 \, .
\end{multline}
Arguing with \eqref{eq:gamma} and the fact that
$\alpha_k \leq \alpha < 1$, we obtain
\begin{equation} \label{eq:gamma2}
\varphi_{k+1} - \varphi_k \ < \
\alpha_k (\varphi_k - \varphi_{k-1}) + 2 \eta_k \, , \ k \geq 0 \, .
\end{equation}

Now, the summability of $\sum_{k\geq0} \eta_k$ (see
Remark~\ref{rem:alpha}) together with \eqref{eq:gamma2}, allow us to
apply Lemma~\ref{lem:appA} to the sequences $(\alpha_k)$,
$(\varphi_k)$, $(\eta_k)$.
This lemma guarantees the existence of $\bar\varphi \in \R$ s.t.
$\varphi_k \to \bar\varphi$ as $k \to \infty$. Consequently, the
boundedness of $(x_k)$ follows.
The boundedness of $(w_k)$ follows from the one of $(x_k)$, together
with \eqref{def:iniT-w} and the fact that $\alpha_k \in [0,\alpha]$.
\medskip

\noindent {\bf Assertion (c):}
We verify the summability of the four series in \eqref{eq:series}.
From \eqref{eq:gamma} follows
\begin{center}
$\norm{\la_k A^*(A x_{k+1} - y)}^2 + 2 \la_k \norm{A x_{k+1} - y}^2
\ < \ \alpha_k \varphi_k + (\varphi_k - \varphi_{k+1})
    + 2 \eta_k \, , \ k \geq 0 \, .$
\end{center}
Adding up this inequality for $k = 0, \ldots, l$  we obtain
\begin{eqnarray}
\D\summ_{k=0}^l \norm{\la_k A^*(A x_{k+1} - y)}^2 +
2 \summ_{k=0}^l \la_k \norm{A x_{k+1} - y}^2
&   <  & \summ_{k=0}^l \alpha_k \varphi_k + (\varphi_0 - \varphi_{l+1})
         + 2 \summ_{k=0}^l \eta_k \nonumber \\
& \leq & M \summ_{k=0}^l \theta_k + \varphi_0 + 2 \summ_{k=0}^l \eta_k \, ,
         \label{eq:series-aux}
\end{eqnarray}
where $(\theta_k)$ is the sequence in Assumption~\ref{ass:alpha} and
$M = \sup_k \varphi_k$ (since $\varphi_k \to \bar\varphi$ as $k \to \infty$,
it holds $M < \infty$).
Taking the limit $l\to\infty$ in \eqref{eq:series-aux} we obtain the
summability of the first two series in \eqref{eq:series}.

The summability of $\sum_k \norm{x_{k+1} - w_k}^2$, the third series
in \eqref{eq:series}, is a consequence of \eqref{def:iniT-x} and the
fact that $\sum_k \norm{\la_k A^*(Ax_{k+1} - y)}^2 < \infty$ (the second
series in \eqref{eq:series}).

Since $\alpha_k < 1$, we argue with \eqref{def:iniT-w} to estimate
\begin{center}
$\norm{x_{k+1} - x_k}^2 \, = \,
\norm{x_{k+1} - w_k + \alpha_k (x_k-x_{k-1})}^2 \, \leq \,
2 \norm{x_{k+1} - w_k}^2 + 2 \norm{x_k - x_{k-1}}^2 , \ k \geq 0 \, .$
\end{center}
Summing up this inequality we obtain
\begin{center}
$\D\summ_{k\geq0} \norm{x_{k+1} - x_k}^2
\ \leq \
2 \summ_{k\geq0} \norm{x_{k+1} - w_k}^2 + 2 \summ_{k\geq0} \eta_k
\ < \ \infty \, ,$
\end{center}
establishing the summability of the last series in \eqref{eq:series}.
\end{proof}

\subsection{The exact data case} \label{ssec:2.2}

In what follows we prove a (strong) convergence result for the iniT method
(Algorithm~\ref{alg:init-exact}) in the case of exact data.
\newpage

\begin{remark} \label{rem:min}
It is well known that problem $Ax = y$ admits an $x_0$-minimal norm solution
\cite{EngHanNeu96}, i.e. there exists an element $x^\dag \in X$ satisfying
$\norm{x^\dag - x_0} = \inf \{ \norm{x - x_0} ; \ Ax = y \}$; notice that
$x^\dag$ is the only solution of \eqref{eq:inv-probl} with this property.
\end{remark}

\begin{remark} \label{rem:range}
It holds $x_k - x_{k-1} \in Rg(A^*)$,%
\footnote{Here $Rg(A^*)$ denotes the range of the operator $A^*$.}
for all $k \in \N$.
Indeed, for $k = 0$ it holds $x_0 - x_{-1} = 0 = A^*(0)$; assume that
$x_j - x_{j-1} \in Rg(A^*)$ holds true for $j = 0, \dots, k$;
it follows from \eqref{def:iniT-w} and \eqref{def:iniT-x} that
$x_{k+1} - x_k = \alpha_k (x_k - x_{k-1}) - \la_k A^*(Ax_{k+1} - y)
\in Rg(A^*)$.
\end{remark}

\begin{theorem}[Convergence for exact data] \label{th:conv}
Let $(x_k)$, $(w_k)$ be sequences generated by Algorithm~\ref{alg:init-exact},
with $(\la_k)$, $(\alpha_k)$ defined as in Steps~[1] and [2.1] respectively.
Moreover, assume that:

% (1) \ $\sum_{k\geq1} \alpha_{k-1} \norm{x_k - x_{k-1}}^2 < \infty$;

(A1) \ The sequence $(\alpha_k)$ satisfies Assumption~\ref{ass:alpha};

(A2) \ $(\alpha_k)$ is monotone non-increasing and $\alpha < \frac12$
    (see Step~[1] of Algorithm~\ref{alg:init-exact});

(A3) \ $\la_k \geq \la > 0$, for $k \geq 0$.

\noindent Then either the sequence $(x_k)$ stops after a finite number
$k_0 \in \N$ of steps (in this case it holds $A x_{k_0} = y$), or it
converges strongly to $x^\dag$ the $x_0$-minimal norm solution of $Ax=y$.
\end{theorem}
\begin{proof}
We have to consider two cases.

\noindent \textbf{Case I:} \ $A{x}_{k_0}= y$ for some $k_0 \in \mathbb{N}$.
\\
In this case, as observed in Remark~\ref{rem:station}, the sequence $(x_k)$
generated by Algorithm~\ref{alg:init-exact} reads $x_{-1}, x_0, \dots, x_{k_0}$,
and it holds $A x_{k_0} = y$.
\medskip

\noindent \textbf{Case II:} \ $A{x}_k\neq y$, for every $k \geq 0$.
\\
Notice that, in this case, the sequence $\big(\norm{A x_k - y}\big) \in \R$ is
strictly positive.
Moreover, it follows from Assumption~(A3) together with Proposition~\ref{prop:aux}~(c)
that $\lim_k \norm{A x_k - y} = 0$.
Therefore, there exists a strictly monotone increasing sequence $(l_j) \in \mathbb{N}$
satisfying
\begin{equation} \label{eq:min}
\|A{x}_{l_j}-y\|\leq \|A{x}_{k}-y\|\,\,\,\, \mbox{for} \,\,\, 0\leq k \leq l_j.
\end{equation}
Next, given $k > 0$ and $z \in X$, we estimate
\begin{eqnarray*}
\|w_k  -  z\|^2 - \|x_{k+1} - z\|^2 &\! = &\! 
      - \norm{x_{k+1} - 
 w_k}^2 - 2 \ipl x_{k+1} - w_k , w_k - z \ipr
      \nonumber \\
&\!\!\! \leq &
      - 2 \ipl x_{k+1} - w_k , w_k - z \ipr
      \nonumber \\
&\!\!\! = &
      2 \ipl \la_k A^*( A x_{k+1} - y) , w_k - z \ipr
      \nonumber \\
&\!\!\! = &
      2 \la_k \ipl A x_{k+1} - y , A(w_k - x_{k+1}) + (A x_{k+1}-y) + (y-Az) \ipr
      \nonumber \\
&\!\!\! \leq &
      2 \ipl \la_k A^* (A x_{k+1} - y) , w_k - x_{k+1} \ipr
      \ + \ 2 \la_k \norm{A x_{k+1}-y}^2 \nonumber \\
& &   + \ 2 \la_k \norm{A x_{k+1}-y} \ \norm{Az-y}
      \nonumber \\
&\!\!\! \leq &
      \norm{\la_k A^*(A x_{k+1} - y)}^2 + \norm{w_k - x_{k+1}}^2
      + 2 \la_k \norm{A x_{k+1} - y}^2 \nonumber \\
& &   + \ 2 \la_k \norm{A x_{k+1}-y} \ \norm{Az-y} .
\end{eqnarray*}
Taking $z = x_{l_j}$ in the last inequality and arguing with \eqref{eq:min},
we obtain
\begin{multline} \label{eq:z}
\norm{w_k - x_{l_j}}^2 - \norm{x_{k+1} - x_{l_j}}^2
\ \leq \\ \norm{\la_k A^*(A x_{k+1} - y)}^2  +  \norm{w_k - x_{k+1}}^2
          + 4 \la_k \norm{A x_{k+1} - y}^2 \ =: \ \mu_k \, ,
\end{multline}
for $0 \leq k \leq l_j - 1$.
On the other hand, we conclude from \eqref{eq:w} (with $x = x_{l_j}$) that
\begin{equation} \label{eq:xlj}
\norm{w_k - x_{l_j}}^2 \ \geq \ \norm{x_k - x_{l_j}}^2
+ \alpha_k \big( \norm{x_k - x_{l_j}}^2 - \norm{x_{k-1} - x_{l_j}}^2 \big)
\, , \ k \geq 0 \, .
\end{equation}
Now, combining \eqref{eq:xlj} with \eqref{eq:z}, and arguing with Assumption
(A2), we conclude that
\begin{eqnarray}
\norm{x_k - x_{l_j}}^2 - \norm{x_{k+1} - x_{l_j}}^2
& \leq &
  \alpha_k \big( \norm{x_{k-1} - x_{l_j}}^2 - \norm{x_k - x_{l_j}}^2 \big) + \mu_k
  \nonumber \\
& \leq &
  \alpha_{k-1} \norm{x_{k-1} - x_{l_j}}^2 - \alpha_k \norm{x_k - x_{l_j}}^2
  + \mu_k \, , \label{eq:txlj}
\end{eqnarray}
for $0 \leq k \leq l_j - 1$.

Let $0 \leq m \leq l_j - 1$. Adding up \eqref{eq:txlj} for $k = m, \ldots, l_j-1$
gives us
$$
\norm{x_m - x_{l_j}}^2 - \norm{x_{l_j} - x_{l_j}}^2
\ \leq \ \alpha_{m-1} \norm{x_{m-1} - x_{l_j}}^2
       - \alpha_{l_j-1} \norm{x_{l_j-1} - x_{l_j}}^2
       + \summ_{k=m}^{l_j-1} \mu_k \, ,
$$
from where we derive
\begin{eqnarray*}
\norm{x_m - x_{l_j}}^2
% & \leq &
%   \norm{x_{l_j} - x_{l_j}}^2 + \alpha_{m-1} \norm{x_{m-1} - x_{l_j}}^2
%   - \alpha_{l_j-1} \norm{x_{l_j-1} - x_{l_j}}^2 + \summ_{k=m}^{l_j-1} \mu_k
%   \nonumber \\
& \leq &
  \alpha_{m-1} \norm{x_{m-1} - x_{l_j}}^2 + \summ_{k=m}^{l_j-1} \mu_k
  \nonumber \\
& \leq &
  2 \alpha_{m-1} \big( \norm{x_{m-1} - x_m}^2 + \norm{x_m - x_{l_j}}^2 \big)
  + \summ_{k=m}^{l_j} \mu_k
  \nonumber\\
& \leq &
  2 \summ_{k=m}^{l_j} \alpha_{k-1} \norm{x_{k-1} - x_k}^2
  + 2\alpha_{m-1} \norm{x_m - x_{l_j}}^2 + \summ_{k=m}^{l_j} \mu_k \, .
\end{eqnarray*}
Consequently, whenever $m < l_j$, it holds
$$
(1-2\alpha_{m-1}) \norm{x_m - x_{l_j}}^2 \ \leq \
2 \summ_{k=m}^{l_j} \alpha_{k-1} \norm{x_k - x_{k-1}}^2 +
\summ_{k=m}^{l_j} \mu_k \, .
$$
Now, defining $\beta := (1-2\alpha_0)^{-1}$, we argue with Assumption~(A2)
to conclude that
\begin{eqnarray} \label{eq:number17}
\norm{x_m - x_{l_j}}^2 &\leq&
  2 \beta \summ_{k=m}^\infty \alpha_{k-1} \norm{x_k - x_{k-1}}^2
  + \beta \summ_{k=m}^\infty \mu_k
  \nonumber \\
& \leq &  \beta \summ_{k=m}^\infty  \norm{x_k - x_{k-1}}^2
        + \beta \summ_{k=m}^\infty \mu_k \, , \ m < l_j
\end{eqnarray}
(notice that $\beta > 0$ due to Assumption (A2)).

Notice that Assumption (A1) together with Proposition~\ref{prop:aux}~(c)
guarantee the summability of both series $\sum_k \mu_k$ and
$\sum_k \norm{x_k - x_{k-1}}^2$.
Therefore, defining $s_m := \beta \sum_{k\geq m}  \norm{x_k - x_{k-1}}^2 + \beta
\sum_{k\geq m} \mu_k$, for $m \in \N$, we have $s_m \to 0$ as $m \to \infty$.

Now let $n > m$ be given. Choosing $l_j > n$ and arguing with \eqref{eq:number17}
we estimate
$$
\norm{x_n - x_m} \ \leq \
\norm{x_n - x_{l_j}} + \norm{x_{l_j} - x_m} \ \leq \
\sqrt{s_n} + \sqrt{s_m} \ \leq \ 2 \sqrt{s_m} \, .
$$
Since $\lim_m s_m = 0$, this inequality allow us to conclude that $(x_k)$ is
a Cauchy sequence.

Consequently, $(x_k)$ converges to some $\overline{x} \in X$.
In order to prove that $\overline{x}$ is a solution of $A x = y$, it suffices
to verify that $\norm{A x_k - y} \to 0$ as $k \to \infty$. This fact, however,
is a consequence of Proposition~\ref{prop:aux}~(c) (see first series in
\eqref{eq:series}) together with Assumption~(A3).

It follows from Remark \ref{rem:range} that $x_{k+1} - x_k \in Rg(A^\ast)
\subset \mathcal{N}(A)^{\perp}$. An inductive argument allow us to conclude
that $\overline{x} \in x_0+\mathcal{N}({A})^{\perp}$.
Thus, from Remark~\ref{rem:min} follows $\overline{x} = x^\dag$.
\end{proof}

\subsection{The noise data case} \label{ssec:2.3}

In what follows we address the noisy data case $\delta > 0$. We begin by
defining the quadratic (square residual) functional $f^\delta: X \ni x \to
f^\delta(x) := \frac12 \norm{Ax-y^\delta}^2 \in \R^+$. The iniT method in the
case noisy data case reads:

\begin{algorithm}[h!]
\begin{center}
\fbox{\parbox{13.5cm}{
$[0]$ choose an initial guess $x_0^\delta \in X$;
      \ set $x_{-1}^\delta := x_0^\delta$; \ $k := 0$;
\medskip

$[1]$ choose $\tau > 1$, \ $\alpha \in [0,1)$ \ and \ $(\la_k)_{k\geq0} \in \R^+$;
\medskip

$[2]$ {\bf while \ $\norm{A x_k^\delta - y^\delta} > \tau \delta$ \ do} 
\smallskip

\ \ \ \ \ $[2.1]$ choose $\alpha_k^\delta \in [0,\alpha]$;
\smallskip

\ \ \ \ \ $[2.2]$ $w_k^\delta := x_k^\delta + \alpha_k^\delta (x_k^\delta-x_{k-1}^\delta)$;
\smallskip

\ \ \ \ \ $[2.3]$ compute $x_{k+1}^\delta \in X$ as the solution of
\smallskip

\centerline{$\la_k \nabla f^\delta(x_{k+1}^\delta) + x_{k+1}^\delta - w_k^\delta \ = \ 0$;}
\smallskip

\ \ \ \ \ $[2.4]$ $k := k+1$;
\smallskip

\ \ \ \ \ {\bf end while}
\medskip

$[3]$ $k^* = k$;
\smallskip
}}
\end{center} \vskip-0.5cm
\caption{Inertial iterated Tikhonov method (iniT) in the noisy data case.}
\label{alg:init-noise}
\end{algorithm}

The stopping criterion used in Algorithm~\ref{alg:init-noise} is based on the
discrepancy principle, i.e. the iteration is stopped at step $k^* =
k^*(\delta, y^\delta)$ satisfying
$$
k^* := \min\{ k\in\N \ ; \ \norm{Ax_k^\delta - y^\delta} \leq \tau\delta \} ,
$$
where $\tau > 1$.
Notice that Algorithm~\ref{alg:init-noise} generate sequences
$(x_k^\delta)_{k=-1}^{k^*}$ and $(w_k^\delta)_{k=0}^{k^*-1}$. In
Proposition~\ref{prop:kstar} we prove that, under appropriate assumptions,
the stopping index $k^*$ in Step~[3] is finite.

In Lemma~\ref{lem:remon} the residuals $\norm{A w_k^\delta - y^\delta}$ and
$\norm{A x_{k+1}^\delta - y^\delta}$ are compared; and in Lemma~\ref{lem:gnoise}
the distances $\norm{w_k^\delta-x^*}$ and $\norm{x_{k+1}^\delta-x^*}$ are
compared (here $x^* \in X$ is a solution of $Ax = y$).

\begin{lemma} \label{lem:remon}
Let $(x_k^\delta)$, $(w_k^\delta)$ be sequences generated by Algorithm~%
\ref{alg:init-noise}. The following assertions hold true
\smallskip
% (a) for $k = 0 , \dots, k^*-1$ we have
% $$
% \norm{A w_k^\delta - y^\delta}^2 - \norm{A x_{k+1}^\delta - y^\delta}^2 =
% \norm{A w_k^\delta - Ax_{k+1}^\delta}^2
% + 2 \ipl A(w_k^\delta - x_{k+1}^\delta) , A x_{k+1}^\delta - y^\delta \ipr;
% $$

\noindent
(a) $\norm{A w_k^\delta - y^\delta}^2 - \norm{A x_{k+1}^\delta - y^\delta}^2
= \norm{A w_k^\delta - Ax_{k+1}^\delta}^2
+ 2 \ipl A(w_k^\delta - x_{k+1}^\delta) , A x_{k+1}^\delta - y^\delta \ipr$,
for $k = 0 , \dots, k^*-1$;
\medskip

\noindent
(b) $\norm{A x_{k+1}^\delta - y^\delta} \leq \norm{A w_k^\delta - y^\delta}$,
for $k = 0, \dots, k^*-1$.
\end{lemma}
\begin{proof}
For $k \leq k^*$ it follows from Step~[2.3] of Algorithm~\ref{alg:init-noise} that
\begin{multline*}
\norm{A w_k^\delta - y^\delta}^2 - \norm{A x_{k+1}^\delta - y^\delta}^2
\ = \
\norm{A w_k^\delta - Ax_{k+1}^\delta}^2
+ 2 \ipl A(w_k^\delta - x_{k+1}^\delta) , A x_{k+1}^\delta - y^\delta \ipr \\
\geq \
\ipl w_k^\delta - x_{k+1}^\delta , A^*(A x_{k+1}^\delta - y^\delta) \ipr
\ = \
\la_k \norm{A^*(A x_{k+1}^\delta - y^\delta)}^2 \ \geq \ 0 \, ,
\end{multline*}
proving assertions (a) and (b).
\end{proof}

\begin{lemma} \label{lem:gnoise}
Let $(x_k^\delta)$, $(w_k^\delta)$ be sequences generated by Algorithm~%
\ref{alg:init-noise}. If $x^*\in X$ is a solution of $A x = y$ then
$$
\norm{w_k^\delta - x^*}^2 - \norm{x_{k+1}^\delta - x^*}^2
\ \geq \ \norm{\la_k A^*(Ax_{k+1}^\delta - y^\delta)}^2
     + 2 \la_k \norm{A x_{k+1}^\delta - y^\delta}
     \Big[ \norm{A x_{k+1}^\delta - y^\delta} - \delta \Big],
$$
for $k = 0, \dots, k^*-1$.
\end{lemma}
\begin{proof}
Let $0 \leq k \leq k^*-1$.
From Step~[2.3] of Algorithm~\ref{alg:init-noise} and \eqref{eq:noisy-i} follow
\begin{eqnarray*}
\|w_k^\delta - x^*\|^2 \!\!\! & \!\!\! - \!\!\! & \!\!\! \|x_{k+1}^\delta-x^*\|^2
 =
  \|w_k^\delta - x_{k+1}^\delta\|^2
  + 2 \ipl w_k^\delta - x_{k+1}^\delta , x_{k+1}^\delta-x^* \ipr \\
& = &
  \| \la_k A^*(Ax_{k+1}^\delta - y^\delta) \|^2
  + 2\la_k \ipl Ax_{k+1}^\delta - y^\delta , Ax_{k+1}^\delta - y
  + y^\delta - y^\delta \ipr \\
& = &
  \| \la_k A^*(Ax_{k+1}^\delta - y^\delta) \|^2
  + 2\la_k \| Ax_{k+1}^\delta - y^\delta\|^2
  + 2\la_k \ipl Ax_{k+1}^\delta - y^\delta , y^\delta - y \ipr \\
& \geq &
  \| \la_k A^*(Ax_{k+1}^\delta - y^\delta) \|^2
  + 2\la_k \| Ax_{k+1}^\delta - y^\delta \|
  \big[ \|Ax_{k+1}^\delta - y^\delta \| - \delta \big] \, ,
\end{eqnarray*}
proving the assertion.
\end{proof}

The next assumption considers a particular choice of the inertial parameter
$\alpha_k^\delta$ in Step~[2.1] of Algorithm~\ref{alg:init-noise}.
It plays a key role in the forthcoming analysis (see Proposition~%
\ref{prop:kstar} and Theorems~\ref{th:stabil} and~\ref{th:semi-conv}).

\begin{assump} \label{ass:alpha-noise}
Given $\alpha \in [0,1)$, define $\alpha_k^\delta$ in Step~[2.1] of
Algorithm~\ref{alg:init-noise} by:
\begin{equation} \label{eq:alphakdelta}
\alpha_0^\delta = \alpha
\quad {\it and} \quad
\alpha^{\delta}_{k} \ := \
        \left\{ \begin{array}{cl}
          \min\left\{ \dfrac{\theta_k}{\|x_k^\delta - x_{k-1}^\delta\|^2} , \
          \theta_k , \ \alpha \right\}
                 & , \ {\rm if } \ \norm{x_k^\delta - x_{k-1}^\delta} > 0 \\[2ex]
          0      & , \ {\rm otherwise}
        \end{array} \right.
\end{equation}
for $k \geq 1$.
Here $(\theta_k)$ is a nonnegative sequence such that $\sum_k \theta_k < \infty$.
\end{assump}

In the sequel we state and prove the main results of this section, namely stability
(Theorem~\ref{th:stabil}) and regularization (Theorem~\ref{th:semi-conv}). First,
however, in the next two remarks we extend to the noisy data case some results
discussed in Lemma~\ref{lemma:aux} and Proposition~\ref{prop:aux}.

\begin{remark} \label{rem:aux-noise}
Arguing as in Lemma~\ref{lemma:aux} we prove that the sequences $(x_k^\delta)$,
$(w_k^\delta)$ generated by Algorithm~\ref{alg:init-noise} (with $(\alpha_k^\delta)$
defined as in Step~[2.1] of this algorithm) satisfy
\begin{equation} \label{eq:w-noise}
\norm{w_k^\delta - x}^2 \ = \ (1+\alpha_k^\delta) \norm{x_k^\delta - x}^2
- \alpha_k^\delta \norm{x_{k-1}^\delta - x}^2
+ \alpha_k^\delta (1+\alpha_k^\delta) \norm{x_k^\delta - x_{k-1}^\delta}^2 ,
\end{equation}
for $x \in X$, and $k = 0, \dots, k^*$.
\end{remark}

\begin{remark} \label{rem:bound-noise}
If the inertial parameters $\alpha_k^{\delta}$ satisfy Assumption~\ref{ass:alpha-noise},
then the sequences $(x_k^\delta)$, $(w_k^\delta)$ generated by Algorithm~%
\ref{alg:init-noise} are bounded.

If $k^*$ in Step~[3] of Algorithm~\ref{alg:init-noise} is finite, the statement
is obvious. Otherwise, define the (infinite) sequences $\varphi_k^\delta :=
\norm{x_k^\delta - x^*}^2$, for $k \geq -1$, and $\eta_k^\delta :=
\alpha_k^\delta \norm{x_k^\delta - x_{k-1}^\delta}^2$, for $k \geq 0$
(here $x^* \in X$ is a solution of $Ax = y$).
Arguing as in the proof of Proposition~\ref{prop:aux}~(b), we apply
Lemma~\ref{lem:appA} to the sequences $(\alpha_k^\delta)$, $(\varphi_k^\delta)$
and $(\eta_k^\delta)$ to conclude that $\varphi_k^\delta$ converges strongly.%
\footnote{Notice that Lemma~\ref{lem:gnoise} is used together with Remark~\ref{rem:aux-noise}
to derive \eqref{eq:gamma} for $x_k^\delta$, $y^\delta$, $\varphi_k^\delta$,
$\alpha_k^\delta$ and $\la_k$; additionally Assumption~\ref{ass:alpha-noise} is used
to guarantee $\sum_k \eta_k^\delta < \infty$.}
The boundedness of $(x_k^\delta)$ follows from this fact.
The boundedness of $(w_k^\delta)$ follows from the one of $(x_k^\delta)$ and
the fact that $\norm{w_k^\delta} = \norm{x_k^\delta + \alpha_k^\delta
(x_k^\delta - x_{k-1}^\delta)}$ for $k \geq 0$.
\end{remark}

\begin{propo} \label{prop:kstar}
Let $(x_k^\delta)$, $(w_k^\delta)$ be sequences generated by Algorithm~%
\ref{alg:init-noise}, and $\alpha_k^\delta$ satisfy Assumption~\ref{ass:alpha-noise}.
If $\sum_k \la_k = \infty$, then the stopping index $k^*$ defined in Step~[3]
is finite.
Moreover, if $\la_k \geq \la > 0$ it holds
$$
k^* \leq \big[ 2 \la \tau \delta^2 (\tau - 1) \big]^{-1}
         \Big( \norm{x_0 - x^*}^2 + M^\delta {\T\sum_k} \, \alpha_k^\delta
         + 2 {\T\sum_k} \, \theta_k \Big) ,
$$
where $x^* \in X$ is a solution of $Ax = y$, $M^\delta := \sup_k \norm{x_k^\delta-x^*}^2$,
and $(\theta_k)$ is the sequence in Assumption~\ref{ass:alpha-noise}.
\end{propo}
\begin{proof}
Assume by contradiction that $k^*$ is not finite. It follows from Lemma~\ref{lem:gnoise}
that
$$
2\la_k \tau \delta^2 (\tau-1)
\ \leq \
2 \la_k \norm{Ax_{k+1}^\delta - y^\delta}
\Big[ \norm{Ax_{k+1}^\delta - y^\delta} - \delta \Big]
\ \leq \
\norm{w_k^\delta - x^*}^2 - \norm{x_{k+1}^\delta - x^*}^2 ,
$$
for $k \geq 0$.
From this inequality and \eqref{eq:w-noise} with $x = x^*$ we obtain
\begin{eqnarray} \label{eq:kfinite}
2\la_k \tau \delta^2 (\tau-1)
& \leq &
  (1 + \alpha_k^\delta) \, \norm{x_k^\delta - x^*}^2
  + \alpha_k^\delta (1 + \alpha_k^\delta) \, \norm{x^\delta_{k} - x^\delta_{k-1}}^2
  - \norm{x_{k+1}^\delta - x^*}^2
  \nonumber\\
& \leq &
  \norm{x_k^\delta - x^*}^2 - \norm{x_{k+1}^\delta - x^*}^2
  + \alpha_k^\delta \, \norm{x^\delta_{k} - x^*}^2
  + 2 \alpha_k^\delta \, \norm{x^\delta_{k} - x^\delta_{k-1}}^2 \nonumber\\
& \leq &
  \norm{x_k^\delta - x^*}^2 - \norm{x_{k+1}^\delta - x^*}^2
  + \alpha_k^\delta \, M^\delta + 2 \theta_k \, ,
\end{eqnarray}
for $k \geq 0$ (to obtain the second inequality we used the fact $\alpha_k^\delta
\leq \alpha < 1$). Adding up \eqref{eq:kfinite} for $k = 0, \cdots, l$ we obtain
\begin{equation} \label{eq:kfinite2}
2 \tau (\tau-1) \delta^2 \summ_{k=0}^l \la_k
\ \leq \
\norm{x_0^\delta - x^*}^2 + M^\delta \summ_{k=0}^l \alpha_k^\delta
+ 2 \summ_{k=0}^{l} \theta_k \, .
\end{equation}
Notice that the right hand side of \eqref{eq:kfinite2} is bounded due to
Assumption~\ref{ass:alpha-noise}.
Consequently, due to the assumption $\sum_k \la_k = \infty$, inequality
\eqref{eq:kfinite2} leads to a contradiction when $l \to \infty$, proving
that $k^*$ is finite.

To prove the second assertion, it is enough to take $l = k^*$ in
\eqref{eq:kfinite2} and argue with the additionally assumption $\la_k \geq \la > 0$.
\end{proof}

\begin{theorem}[Stability] \label{th:stabil}
Let $(\delta^j)_j$ be a zero sequence, and $(y^{\delta^j})_j$ be a corresponding
sequence of noisy data satisfying \eqref{eq:noisy-i} for some $y \in Rg(A)$.
For each $j\in\N$, let $(x_l^{\delta^j})_{l=-1}^{k^*_j}$ and
$(w_l^{\delta^j})_{l=0}^{k^*_j-1}$ be sequences generated by
Algorithm~\ref{alg:init-noise}, with $(\alpha_l^{\delta^j})_{l=0}^{k^*_j}$
satisfying Assumption~\ref{ass:alpha-noise} (here $k^*_j =
k^*(\delta^j, y^{\delta^j})$ are the corresponding stopping indices defined
in Step~[3]).
Moreover, let $(x_k)$, $(w_k)$ be the sequences generated by Algorithm~%
\ref{alg:init-exact} (with $(\alpha_k)$ satisfying Assumption~\ref{ass:alpha}).
Then, either the sequences $(x_k)$, $(w_k)$ are not finite and
\begin{equation} \label{eq:stabil}
\lim_{j\to\infty} x_k^{\delta^j} = x_k
\, , \ \ \
\lim_{j\to\infty} w_k^{\delta^j} = w_k \, , \ for \ k = 0 , 1 , \dots
\end{equation}
or the sequences $(x_k) = (x_k)_{k=-1}^{k_0}$, $(w_k) = (w_k)_{k=0}^{k_0-1}$
are finite and it holds
\begin{equation} \label{eq:stabil-f}
\lim_{j\to\infty} x_k^{\delta^j} = x_k \, , \ 0 \leq k \leq k_0
\ \ \ \mbox{and} \ \ \
\lim_{j\to\infty} w_k^{\delta^j} = w_k \, , \ 0 \leq k \leq k_0-1 \, ,
\end{equation}
for some $k_0 \in \N$ (in this case, $x_{k_0} \in X$ is a solution of $Ax = y$).
\end{theorem}
\begin{proof}
We present here a proof by induction.
Notice that $w_0^{\delta^j} = w_0 = x_0 = x_0^{\delta^j}$ and $\alpha_0^{\delta^j}
= \alpha_0 = \alpha$, for all $j \in \N$. Thus, \eqref{eq:stabil} holds for
$k = 0$.
Next, assume the existence of $(x_{l})_{l\leq k}$ and $(w_{l})_{l\leq k}$
generated by Algorithm~\ref{alg:init-exact} (with $(\alpha_l)_{l\leq k}$ satisfying
Assumption~\ref{ass:alpha}) such that $\lim_j x_l^{\delta^j} = x_l$
and $\lim_j w_l^{\delta^j} = w_l$, for $l = 0 , \dots , k$.

Define $x_{k+1} := (\la_k A^*A + I)^{-1} (w_k + \la_k A^* y)$ as in Step~[2.2] of
Algorithm~\ref{alg:init-exact}. Thus, from Step~[2.3] of Algorithm~\ref{alg:init-noise}
follows
\begin{eqnarray*}
x_{k+1}^{\delta^j} - x_{k+1}
% &\!\!\! = &\!\!\!
%   w_k^{\delta^j} - w_k - \la_k A^*(A x_{k+1}^{\delta^j} - y^{\delta^j})
%   + \la_k A^* (A x_{k+1} - y) \\
&\!\!\! = &\!\!\!
  w_k^{\delta^j} - \la_k A^*(A x_{k+1}^{\delta^j} - y^{\delta^j})
- \big[ w_k - \la_k A^* (A x_{k+1} - y) \big] \\
&\!\!\! = &\!\!\!
  w_k^{\delta^j} - w_k - \la_k A^*A(x_{k+1}^{\delta^j} - x_{k+1})
  + \la_k A^*(y^{\delta^j} - y) \, .
\end{eqnarray*}
Consequently, $\norm{(I+\la_k A^*A)(x_{k+1}^{\delta^j}-x_{k+1})}
\leq \norm{w_k^{\delta^j} - w_k} + \norm{A} \la_k \delta^j$. This inequality
together with the inductive hypothesis and the fact that $\lim_j \delta^j = 0$,
allow us to estimate
\begin{equation} \label{eq:k1}
\lim_{j\to\infty} \norm{(I+\la_k A^*A)(x_{k+1}^{\delta^j}-x_{k+1})}
\ \leq \
\lim_{j\to\infty} \norm{w_k^{\delta^j} - w_k}
+ \norm{A} \la_k \lim_{j\to\infty} \delta^j \ = \ 0 \, .
\end{equation}
Since $\norm{(I+\la_k A^*A)} \geq 1$, we conclude from \eqref{eq:k1} that
$\lim_j \norm{x_{k+1}^{\delta^j} - x_{k+1}} = 0$.
\medskip

At this point, we have to consider two complementary cases:
\medskip

\noindent {\bf Case~1.} \ $\norm{A x_{k+1} - y} = 0$.
\ In this case, Algorithm~\ref{alg:init-exact} stops (see Remark~\ref{rem:station}).
Consequently, \eqref{eq:stabil-f} holds true with $k_0 = k+1$ (it is immediate to see
that $A x_{k_0} = y$).
\medskip

\noindent {\bf Case~2.} \ $\norm{A x_{k+1} - y} > 0$.
\ In this case we must consider two scenarios:

{\bf (2a)} $x_{k+1} \not= x_k$.
Choose $\alpha_{k+1}$ in agreement with Assumption~\ref{ass:alpha}, i.e.
\begin{equation} \label{eq:k3}
\alpha_{k+1} \ := \
\min \big\{ \theta_{k+1} \norm{x_{k+1}-x_k}^{-2}, \theta_{k+1}, \alpha \big\} > 0 ,
\end{equation}
and define $w_{k+1}$ as in Step~[2.1] of Algorithm~\ref{alg:init-exact}, i.e.
$w_{k+1} := x_{k+1} + \alpha_{k+1}(x_{k+1}-x_k)$.
%
% If $\norm{A x_{k+1} - y} = 0$, then $\alpha_{k+1} = 0$ and $w_{k+1} = x_{k+1}$
% (see Step~[2.1] of Algorithm~\ref{alg:init-exact}).
% Since $x_k^{\delta^j} \to x_k$ and $x_{k+1}^{\delta^j} \to x_{k+1}$ as $j\to\infty$,
% and $(\alpha^{\delta^j}_{k+1})_j \in [0,\alpha]$.
% Thus, \eqref{eq:lim-wk} implies $\lim_j w_{k+1}^{\delta^j} = w_{k+1}$.
%
Since $x_k^{\delta^j} \to x_k$ and $x_{k+1}^{\delta^j} \to x_{k+1}$ as $j \to
\infty$, it follows from Assumption~\ref{ass:alpha-noise} and \eqref{eq:k3} that
$\alpha^{\delta^j}_{k+1} \to \alpha_{k+1}$ as $j\to\infty$. Consequently,
\begin{equation} \label{eq:lim-wk1}
\lim_{j\to\infty} w_{k+1}^{\delta^j} \ = \
\lim_{j\to\infty} \big( x_{k+1}^{\delta^j}
  + \alpha_{k+1}^{\delta^j} (x_{k+1}^{\delta^j} - x_k^{\delta^j}) \big)
\ = \ x_{k+1} + \alpha_{k+1}(x_{k+1} - x_k) \ = \ w_{k+1} \, .
\end{equation}

{\bf (2b)} $x_{k+1} = x_k$. Choose $\alpha_{k+1} := 0$, in agreement with Assumption~%
\ref{ass:alpha}, and define $w_{k+1}$ as in Step~[2.1] of Algorithm~\ref{alg:init-exact},
i.e. $w_{k+1} := x_{k+1} + \alpha_{k+1}(x_{k+1} - x_k) = x_{k+1}$.
Since $\lim_j x_k^{\delta^j} = x_k$, $\lim_j x_{k+1}^{\delta^j} = x_{k+1}$,
and $(\alpha_{k+1}^{\delta^j})_j$ is bounded (see Assumption~\ref{ass:alpha-noise}),
it follows from Step~[2.2] of Algorithm~\ref{alg:init-noise} that
\begin{equation} \label{eq:lim-wk2}
\lim_{j\to\infty} w_{k+1}^{\delta^j} \ = \
\lim_{j\to\infty} \big( x_{k+1}^{\delta^j}
  + \alpha_{k+1}^{\delta^j} (x_{k+1}^{\delta^j} - x_k^{\delta^j}) \big)
\ = \ x_{k+1} \ = \ w_{k+1} \, .
\end{equation}

Thus, it follows from \eqref{eq:lim-wk1} and \eqref{eq:lim-wk2} that, in Case~2,
$w_{k+1}^{\delta^j} \to w_{k+1}$ as $j\to\infty$. This completes the inductive
proof. Consequently, in Case~2, the assertions in \eqref{eq:stabil} hold true.
\end{proof}

\begin{theorem}[Semi-convergence] \label{th:semi-conv}
Let $(\delta^j)_j$ be a zero sequence, $(y^{\delta^j})_j$ be a corresponding
sequence of noisy data satisfying \eqref{eq:noisy-i} for some $y \in Rg(A)$,
and assume that (A1), (A2) and (A3) in Theorem~\ref{th:conv} hold.
For each $j\in\N$, let $(x_l^{\delta^j})_{l=-1}^{k^*_j}$ and
$(w_l^{\delta^j})_{l=0}^{k^*_j-1}$ be sequences generated Algorithm~\ref{alg:init-noise},
with $(\alpha_l^{\delta^j})_{l=0}^{k^*_j}$ satisfying Assumption~\ref{ass:alpha-noise}
(here $k^*_j = k^*(\delta^j, y^j)$ are the corresponding stopping indices
defined in Step~[3] of Algorithm~\ref{alg:init-noise}).
\\
Then, the sequence $(x_{k^*_j}^{\delta^j})_j$ converges strongly to $x^\dag$,
the $x_0$-minimal norm solution of $Ax = y$.
\end{theorem}
\begin{proof}
It suffices to prove that every subsequence of $(x_{k^*_j}^{\delta^j})_j$ has
itself a subsequence converging strongly to $x^\dag$. In what follows, we denote
a subsequence of $(x_{k^*_j}^{\delta^j})_j$ again by $(x_{k^*_j}^{\delta^j})_j$. 

Let $(x_k)$, $(w_k)$ be sequences generated by Algorithm~\ref{alg:init-exact}
with exact data $y$ and $(\alpha_k)$ as in the theorem assumptions.
Two cases are considered:
\medskip

\noindent {\bf Case~1.} \
The corresponding subsequence $(k^*_j)_j \in \N$ has a finite accumulation point. \\
In this case, we can extract a subsequence $(k^*_{j_m})$ of $(k^*_j)$
such that $k^*_{j_m} = n$, for some $n \in \N$ and all indices $j_m$.
Applying Theorem~\ref{th:stabil} to $(\delta^{j_m})$ and $(y^{\delta^{j_m}})$,
we conclude that $x_{k^*_{j_m}}^{\delta^{j_m}} = x_n^{\delta^{j_m}} \to x_n$,
as $j_m \to \infty$. We claim that $A x_n = y$. Indeed, notice that
$\norm{A x_n - y} = \lim_{j_m\to\infty} \norm{A x_{n}^{\delta^{j_m}} - y} \leq
\lim_{j_m\to\infty} \big( \norm{A x_{n}^{\delta^{j_m}} - y^{\delta^{j_m}}} +
\norm{y^{\delta^{j_m}} - y} \big) \leq \lim_{j_m\to\infty} (\tau+1) \delta^{j_m} = 0$.
\medskip

% \begin{eqnarray*}
% \norm{A x_n - y}
% &  =   &
%   \lim_{j_m\to\infty} \norm{A x_{k^*_{j_m}}^{\delta^{j_m}} - y}
% \ \leq \
%   \lim_{j_m\to\infty} \big( \norm{A x_{n}^{\delta^{j_m}} - y^{\delta^{j_m}}}
%   + \norm{y^{\delta^{j_m}} - y} \big) \\
% & \leq &
%       \lim_{j_m\to\infty} (\tau+1) \delta^{j_m} = 0 .
% \end{eqnarray*}

\noindent {\bf Case~2.} \
The corresponding subsequence $(k^*_j)_j$ has no finite accumulation point. \\
In this case we can extract a monotone strictly increasing subsequence, again
denoted by $(k^*_j)_j$, such that $k^*_j \to \infty$ as $j \to \infty$.

Take $\varepsilon > 0$. From Theorem~\ref{th:conv} follows the existence of
$K_1 = K_1(\varepsilon) \in \N$ such that
\begin{equation} \label{eq:eps3-1}
\norm{x_k - x^\dag} \ < \ \T\frac13 \varepsilon \, , \ k \geq K_1 \, .
\end{equation}
Since $\sum_k \theta_k$ is finite (see Assumption~\ref{ass:alpha-noise}),
there exists $K_2 = K_2(\varepsilon) \in \N$ such that
\begin{equation} \label{eq:eps3-2}
\summ_{k\geq K_2} \theta_k \ \leq \ \frac13 \varepsilon \, .
\end{equation}
Define $K = K(\varepsilon) := \max\{ K_1 , K_2 \}$.
Due to the monotonicity of the subsequence $(k^*_j)_j$, there exists
$J_1 \in \N$ such that $k^*_j \geq K$ for $j \geq J_1$.

Theorem~\ref{th:stabil} applied to the subsequences $(\delta^j)_j$,
$(y^{\delta^j})_j$ (corresponding to the subsequence $(k^*_j)_j$)
implies the existence of $J_2 \in \N$ s.t.
\begin{equation} \label{eq:eps3-3}
\norm{x_K^{\delta^j} - x_{K}} \ \leq \ \T\frac13 \varepsilon \, , \ j > J_2 \, .
\end{equation}
Set $J := \max \{ J_1, J_2 \}$.
From Lemma~\ref{lem:gnoise} (with $x^* = x^\dag$) and Step~[2.2] of
Algorithm~\ref{alg:init-noise} follow
$$
\norm{x_{k+1}^{\delta^j} - x^\dag} \ \leq \ \norm{w_k^{\delta^j} - x^\dag}
\ \leq \ \norm{x_k^{\delta^j} - x^\dag} + \alpha_k^{\delta^j}
         \norm{x_k^{\delta^j} - x_{k-1}^{\delta^j}} \, ,
$$
for $j \geq J$ and $k = 0, \dots, k^*_j - 1$. Consequently,
\begin{equation} \label{eq:telescopic}
\norm{x_{k+1}^{\delta^j} - x^\dag} - \norm{x_k^{\delta^j} - x^\dag}
\ \leq \ \sqrt\alpha_k^{\delta^j} \, \sqrt\alpha_k^{\delta^j}
         \norm{x_k^{\delta^j} - x_{k-1}^{\delta^j}}
\ \leq \ \T\frac12 \alpha_k^{\delta^j} + \frac12 \alpha_k^{\delta^j}
         \norm{x_k^{\delta^j} - x_{k-1}^{\delta^j}}^2 ,
\end{equation}
for $j \geq J$ and $k = 0, \dots, k^*_j - 1$. 
Now, adding \eqref{eq:telescopic} for $k = K, \dots, k^*_j-1$ we obtain
$$
\norm{x_{k^*_j}^{\delta^j} - x^\dag} \ \leq \
  \norm{x_K^{\delta^j} - x^\dag}
  \ + \ \T\frac12 \summ_{k=K}^{k^*_j-1} \alpha_k^{\delta^j}
  \ + \ \frac12 \summ_{k=K}^{k^*_j-1} \alpha_k^{\delta^j}
        \norm{x_k^{\delta^j} - x_{k-1}^{\delta^j}}^2 ,
  \ j \geq J \, .
$$
Thus, arguing with Assumption~\ref{ass:alpha-noise}, together with
\eqref{eq:eps3-1}, \eqref{eq:eps3-2} and \eqref{eq:eps3-3}, we obtain
\begin{eqnarray*}
\norm{x_{k^*_j}^{\delta^j} - x^\dag}
& \leq &
  \norm{x_K^{\delta^j} - x^\dag}
  \ + \ \T\frac12 \summ_{k=K}^{k^*_j-1} \theta_k
  \ + \ \frac12 \summ_{k=K}^{k^*_j-1} \theta_k \nonumber
\\
& \leq &
  \norm{x_K^{\delta^j} - x_K} \ + \ \norm{x_K - x^\dag}
  \ + \ \summ_{k\geq K} \theta_k \, ,
\\
& \leq &
  \T\frac13\varepsilon + \frac13\varepsilon + \frac13\varepsilon
  \, , \ j \geq J \, .
\end{eqnarray*}

Repeating the above argument for $\varepsilon = 1, \frac12, \frac13, \dots$
we are able to generate a sequence of indices $j_1 < j_2 < j_3 \dots$ such that
$$
\norm{x_{k^*_{j_m}}^{\delta^{j_m}} - x^\dag} \ \leq \ \frac1m \, , \ m \in \N \, .
$$
This concludes Case~2, and completes the proof of the theorem.
\end{proof}

% ------------------------------------------------------------------------
\section{Numerical experiments} \label{sec:numerics}

In this section the {\em Inverse Potential Problem} \cite{FSL05, CLT09, HR96, DAL09}
and the {\em Image Deblurring Problem} \cite{BeBo98, Ber09} are used to test the
numerical efficiency of the iniT method.
All computations are performed using MATLAB\textsuperscript{\textregistered}\,R2017a,
running on an Intel\textsuperscript{\textregistered}\,Core\textsuperscript{\tiny TM}%
\,i9-10900 CPU.

\subsection{The Inverse Potential Problem} \label{ssec:num-ipp}

The underlying forward problem is as follows. Let $\Omega \subset \R^2$ be
bounded with Lipschitz boundary. For a given source $x \in L_2(\Omega)$, consider
the elliptic boundary value problem (BVP) of finding $u$ such that
\begin{equation} \label{eq:ipp}
- \Delta u \, = \, x \, , \ {\rm in} \ \Omega \, , \quad
u \, = \, 0          \, , \ {\rm on} \ \partial\Omega .
\end{equation}
weakly.

The corresponding inverse problem is the Inverse Potential Problem (IPP).
It consists of recovering an $L_2$--function $x$, from measurements of the Neumann
trace of its corresponding potential $u \in H^1(\Omega)$ on the boundary of $\Omega$,
i.e. we aim to recover $x \in L_2(\Omega)$ from the available data $y := u_{\nu}
|_{\partial\Omega} \in H^{-1/2}(\partial\Omega)$.%
\footnote{Here $u_{\nu}$ denotes the outward normal derivative of $u$ on
$\partial\Omega$.}

For issues related to 'redundant data' and 'identifiability' in IPP we refer the
reader to \cite{CLT09} and the references therein.
Generalizations of this linear inverse problem lead to distinct applications, e.g.,
Inverse Gravi\-metry \cite{Isa06, DAL09}, EEG \cite{EF06}, and EMG \cite{DAP08}.

The linear direct problem is modeled by the operator $A: L_2(\Omega) \to H^{-1/2}
(\partial\Omega)$, which is defined by $A x := u_{\nu} |_{\partial\Omega}$, where
$u \in H_0^1(\Omega)$ is the unique weak solution of \eqref{eq:ipp} (for the
solution theory of this particular BVP we refer the reader to \cite{GiTr98, HR96}).
Using this notation, the IPP can be modeled in the form \eqref{eq:inv-probl},
where the available noisy data $y^\delta \in H^{-1/2}(\partial\Omega)$ satisfies
\eqref{eq:noisy-i}.

\subsubsection*{Discretization of the direct problem}

The discretization of the BVP~\eqref{eq:ipp} relies on finite element techniques.
Assume that $x\in L^2(\Omega)$ is piecewise constant, i.e., $x=\sum_{i=1}^N x_i\chi_i$,
where $\chi_i(\cdot)$ is the characteristic function of the element $K_i \subset \Omega$.
Those elements define a partition of $\Omega$ in the sense that $K_i\cap K_j=\emptyset$
for $i\ne j$, and $\bar\Omega=\cup_{i=1}^N\bar K_i$. As a preprocessing step, we
determine $\gamma_i = (u_i)_\nu |_{\partial\Omega}$ for $i = 1,\dots,N$, where
\begin{equation} \label{eq:ipp_i}
-\Delta u_i \, = \, \chi_i \,  \ {\rm in} \ \Omega \, , \quad
u_i \, = \, 0         \, \ {\rm on} \ \partial\Omega\, .
\end{equation}
Then, $\gamma=\sum_{i=1}^N x_i \gamma_i$.

Of course, solving~\eqref{eq:ipp_i} exactly is not feasible, and a further
discretization step is necessary. Discretizing $\Omega$ into triangular finite
elements of maximum length $h \ll \operatorname{diameter} K_i$ for $i = 1,\dots,N$,
and using a primal-hybrid finite element discretization~\cite{RT77,AHPV13,MS21},
we compute $\gamma_i^h$ as an approximation of $\gamma_i$.

\subsubsection*{Experiments with noisy data (IPP)}

The numerical experiments discussed in this section follow \cite{BLS20,CLT09}.
Here, $\Omega = (0,1) \times (0,1)$ and the unknown ground
truth $x^\star$ is assumed to be an $L_2$--function (see Figure~\ref{fig:IPP-setup}).
The iniT method is compared with the iT method. The setup of our numerical
experiments is detailed as follows:
\bigskip

--- Solve problem~\eqref{eq:ipp} with $x = x^\star$, and compute the exact data $y$.
    \smallskip

--- Add $0.1\%$ and $5\%$ of uniformly distributed random noise to the exact data,
    generating

    \ \ \ \ the noisy data $y^\delta$. \smallskip

--- Use the constant function $x_0 = 1.5$ as initial guess for the iT and iniT methods.
    \smallskip

--- Employ the discrepancy principle with $\tau=1.5$ as the stopping criteria for the
    iT and

    \ \ \ \ iniT methods. \smallskip

--- $\la_k = (\frac23)^k$ for both methods.
    \smallskip

--- Choose $\alpha_k^\delta$ as in Assumption~\ref{ass:alpha-noise} for the iniT method
    with \smallskip
    
    \centerline{ $\theta_k = (1/k)^{1.1}$. } \smallskip

--- Compute $x_{k+1}^\delta$ in Step~[2.3] of Algorithm~\ref{alg:init-noise};
    {\color{black} in each iteration $k = 1, \dots k^*(\delta)$ the {\em Conju-}

    \ \ \ \ {\em gate Gradient} (CG) method \cite{GVL13, Hac91}, MATLAB routine with
    tolerance $10^{-6}$, is used to
    
    \ \ \ \ compute the step \,$s_k^\delta := x_{k+1}^\delta - x_k^\delta$.
    \smallskip

--- In the iT method, obtain $x_{k+1}^\delta$ by solving $(\la_k A^*A + I)
(x - x_k^\delta) = \la_k A^* (y^\delta - A x_k^\delta)$ using
    
    \ \ \ \ the MATLAB CG-routine with tolerance $10^{-6}$.}
    \bigskip

\noindent {\bf Numerical test~1 ($\mathbf{0.1\%}$ noise):}
The iniT method and the iT method are implemented for solving the IPP under the above
described setup, reaching the stopping criteria after 17 and 19 steps respectively.

In Figure~\ref{fig:IPP-setup} the results obtained by the iniT method are presented.
The stopping criterion is reached after $k^*(\delta) = 17$ steps.
The top figure shows the exact solution, the center figure displays the approximate
solution $x_{17}^\delta$; the relative iteration error $|x_{17}^\delta - x^\star|
/ |x^\star|$ is depicted at the bottom figure. Note that the quality of the
reconstruction is not as good close to the discontinuity curve of $x^\star$,
improving farther from it.

The progress of the corresponding {\em relative iteration error}
$\norm{x_k^\delta - x^\star} / \norm{x^\star}$ and {\em relative residual}
$\norm{A x_k^\delta - y^\delta} / \norm{y^\delta}$ are depicted in
Figure~\ref{fig:IPP-evolution1}.
In each of the subplots, we display the iT method in black, and the iniT method in red.
Note the presence of a third curve, in blue. That corresponds to fixing $\alpha_k=2/3$
constant, a choice not covered by our theory (see \emph{Numerical tests 3} below for
further discussion).
In Table~\ref{tab:CG-01pc} the number of CG-steps evaluated at each iteration of the
methods in Figure~\ref{fig:IPP-evolution1} is compared.

\begin{table}[ht]
\begin{center} \begin{tabular}{l c c c c c c c c c c c c c c c c c c c}
\cline{2-20}
{} & \multicolumn{19}{c}{Iteration number} \\
\cline{2-20}
{} &\!1 &\!2 &\!3 &\!4 &\!5 &\!6 &\!7 &\!8 &\!9 &\!\!10&\!\!11&\!\!12&\!\!13&\!\!14&\!\!15&\!\!16&\!\!17&\!\!18&\!\!19 \\
\hline
iT  &\!3&\!3&\!3&\!3&\!3&\!4 &\!4&\!5&\!5&\!\!6&\!\!6&\!\!7&\!\!8&\!\!9&\!\!11&\!\!12&\!\!14&\!\!16&\!\!18 \\
iniT ($\alpha_k = \frac23$)
    &\!3&\!3&\!3&\!3&\!3&\!4&\!4&\!5&\!5&\!\!6 &\!\!6&\!\!7&\!\!8&\!\!9&\!\!10&\!\!12&\!\!13& & \\
iniT &\!\!3&\!\!3&\!\!3&\!\!3&\!\!3&\!\!4&\!\!4&\!\!5&\!\!5&\!\!6&\!\!6&\!\!7&\!\!8&\!\!9&\!\!11&\!\!13&\!\!14& & \\
\hline
\end{tabular}\end{center}
\vskip-0.3cm
\caption{\small Noise level $0.1\%$. Number of CG-steps required to compute
$x_k^\delta$ in each step of the methods presented in Figure~\ref{fig:IPP-evolution1}.}
\label{tab:CG-01pc}
\end{table}

  The accumulated number of CG-steps of these methods read:
% The accumulated number of CG-steps and the execution time of these methods read:

\rule{14.5cm}{1pt} \\ \centerline{iT -- 140 CG-steps \quad\quad
iniT ($\alpha_k = \frac23$) -- 104 CG-steps \quad\quad
iniT -- 107 CG-steps}
%
% \centerline{\ \ \ \ 1.75 seconds \quad\quad\quad\quad\quad
% \ \ \ \ \ \ \ \ \ \ \ \ \ \ 1.37 seconds \quad\quad
% \ \ \ \ \ \ \ \ \ \ 1.40 seconds}
%
\vskip-1.5ex \rule{14.5cm}{1pt}
\medskip

\noindent Notice that the iT method require $30\%$ more CG-steps than the
iniT method. \medskip

\noindent {\bf Numerical test~2 ($\mathbf{5.0\%}$ noise): }
Similarly to the previous case, we present in Figure~\ref{fig:IPP-evolution2}
how the iniT and iT methods perform under a $\mathbf{5.0\%}$ noise level.
The stopping criteria are reached after 9 and 12 steps respectively. The
subplot display the corresponding \emph{relative iteration error} (top) and
\emph{relative residual} (bottom). The curves in black and red correspond to
the iT and iniT methods. The blue curve corresponds to a constant $\alpha_k=2/3$
case, not covered by the theory (see \emph{Numerical tests 3} for further
discussion).
In Table~\ref{tab:CG-50pc} the number of CG-steps needed in each iteration of
the methods presented in Figure~\ref{fig:IPP-evolution2} is compared.

\begin{table}[ht]
\begin{center} \begin{tabular}{l c c c c c c c c c c c c c c c c c c c}
\cline{2-13}
{} & \multicolumn{12}{c}{Iteration number} \\
\cline{2-13}
{} & 1 & 2 & 3 & 4 & 5 & 6 & 7 & 8 & 9 & 10& 11& 12 \\
\hline
iT                 & 3& 3& 3& 3& 3& 4 & 4& 5 & 5& 6 & 7 & 7 \\
iniT
($\alpha_k = 0.8$) & 3& 3& 3& 3& 3& 4 & 4& 5 & 5&   &   &   \\
iniT               & 3& 3& 3& 3& 4& 4 & 4& 5 & 5&   &   &   \\
\hline
\end{tabular}\end{center}
\vskip-0.3cm
\caption{\small Noise level $5.0\%$. Number of CG-steps required to compute
$x_k^\delta$ in each step of the methods presented in Figure~\ref{fig:IPP-evolution2}.}
\label{tab:CG-50pc}
\end{table}

The accumulated number of CG-steps of these methods read:
% The accumulated number of CG-steps and the execution time of these methods read:
\smallskip

\rule{14.5cm}{1pt} \\ \centerline{iT -- 53 CG-steps \quad\quad
iniT ($\alpha_k = 0.8$) -- 33 CG-steps \quad\quad
iniT -- 34 CG-steps}

% \centerline{\ \ \ \ \ \ \ 0.80 seconds \quad\quad\quad\quad\quad
% \ \ \ \ \ \ \ \ \ \ \ \ \ \ 0.53 seconds \quad\quad
% \ \ \ \ \ \ \ \ \ 0.54 seconds}
\vskip-1.5ex \rule{14.5cm}{1pt}
\medskip

\noindent Notice that the iT method require $56\%$ more CG-steps than the
iniT method. \medskip

\noindent {\bf Numerical test~3 (experimenting with constant $\mathbf{\alpha_k^\delta}$):}
Observing the choice of the scaling parameters in Nesterov's accelerated
forward-backward scheme \eqref{eq:nesterov}, a natural question arises:
``How does the iniT method performs if one chooses $\alpha_k^\delta$ constant
in Algorithm~\ref{alg:init-noise}?''

Differently from the explicit two-point type methods (which include Nesterov's
scheme), in the implicit two-point methods (or inertial methods) one must assume
that $\sum_k \alpha_k < \infty$ and
$\sum_k \alpha_k \norm{x_k - x_{k-1}}^2 < \infty$. Indeed, these hypothesis
are needed in the proof of the convergence Theorem~\ref{th:conv} (they are
used to obtain \eqref{eq:series}). As a matter of fact, the summability of the
series $\sum_k \alpha_k \norm{x_k - x_{k-1}}^2$ is quintessential in our analysis:
it is used to prove boundedness of the sequence $(x_k)$ (see Proposition~%
\ref{prop:aux}~(b)).%
\footnote{An analog assumption is used in \cite[Theor.2.1]{AA01} to prove weak
convergence of the inertial proximal method.}

Thus, choosing $\alpha_k^\delta$ constant may not lead to bounded iterations.
Nevertheless, we implemented the iniT method for different (constant) values of
$\alpha_k^\delta$, ranging in the interval (0,1).
In Figures~\ref{fig:IPP-AlphaConst1} and~\ref{fig:IPP-AlphaConst2} we revisit
the noise level scenarios of $0.1\%$ and $5.0\%$ respectively.

For the noise level of $0.1\%$, the best result was obtained for $\alpha_k^\delta
= 2/3$, while $\alpha_k^\delta = 8/10$ was the best choice for the
noise level of $5.0\%$.
For comparison purposes, the iniT method results with these optimal choices of
(constant) $\alpha_k^\delta$ are presented in Figures~\ref{fig:IPP-evolution1}
and~\ref{fig:IPP-evolution2} respectively (blue curves). Notice that:
\smallskip

--- For (constant) $\alpha_k^\delta$ close to zero, the iniT method performs
    similarly as the iT method; \smallskip

--- For (constant) $\alpha_k^\delta$ close to one, the iniT method becomes
    unstable; \smallskip

--- The iniT method with constant choice $\alpha_k^\delta = 2/3$ performs
    similarly as the iniT
    
    \ \ \ \ method with $\alpha_k^\delta$ satisfying Assumption~\ref{ass:alpha-noise}.
    \medskip

\noindent {\bf Numerical test~4 (comparison with explicit two-point type methods):}
The Nesterov's scheme \eqref{eq:nesterov} and the FISTA method \cite{BeTe09} are
two well known methods for solving linear systems of the form \eqref{eq:inv-probl}.
Another natural question that arises is: ``How does the iniT method performance
compare with performance of these established explicit two-point methods?''

To answer this question we revisit the $0.1\%$ noise scenario. In Figure~%
\ref{fig:IPP-explicit} the iniT method (implicit), the Nesterov and the FISTA
methods (both explicit) are implemented for solving the IPP. In the Nesterov
and FISTA methods we use $\gamma = \norm{A^*A}^{-1/2}$. Moreover, in the Nesterov
method we set (the frequently used value) $\alpha = 3.0$.

Due to the distinct nature (implicit/explicit) of these methods, the numerical
effort cannot be compared by simply computing the number of iterations necessary to
reach the stopping criterion. Instead, we compute the execution time of these methods:
\smallskip

\rule{14.5cm}{1pt} \\ \centerline{iniT -- 1.40 seconds \quad\quad
Nesterov -- 6.61 seconds \quad\quad
FISTA -- 6.76 seconds} \vskip-1.5ex \rule{14.5cm}{1pt}
\medskip

% \noindent The execution time of Nesterov and FISTA methods include the
% computation of $\gamma$ as above. If one neglects this computation, the
% execution time becomes: Nesterov -- 0.50 seconds, FISTA -- 0.49 seconds.
% \medskip

\subsection{The Image Deblurring Problem} \label{ssec:num-idp}

The here considered Image Deblurring Problem (IDP) is an ill-posed inverse problem
modeled by a finite-dimensional linear system of the form \eqref{eq:inv-probl}.
In this problem $x^\star \in X = R^n$ are the pixel values of a true image,
while $y \in Y = X$ represents the pixel values of the observed (blurred) image.
In real situations, blurred (noisy) data $y^\delta \in Y$ satisfying
\eqref{eq:noisy-i} are available and the noise level $\delta > 0$ is not
always known.

In this setting the matrix $A$ in \eqref{eq:inv-probl} describes the blurring
process, which is modeled by a space invariant point spread function (PSF).
In the continuous model, the blurring phenomena is modeled by a convolution
operator and $A$ corresponds to an integral operator of the first kind
\cite{Gr84, EngHanNeu96}.
In our model, the discrete convolution is evaluated by means of the Fast Fourier
Transform (FFT) algorithm. We added to the exact data, i.e. the true image
convoluted with the PSF, a normally distributed noise with zero mean and suitable
variance.

In our numerical implementations we follow \cite{BLS20} in the experimental
setup, see Figure~\ref{fig:ID-setup}: (LEFT) True image $x^\star \in R^n$,
with $n = 2^{16}$ (Cameraman image $256 \times 256$);
(CENTER) PSF is the rotationally symmetric Gaussian low-pass filter of dimension
$[257 \times 257]$ and standard deviation $\sigma = 4$; (RIGHT) Blurred image
$y = Ax \in R^n$.

\subsubsection*{Experiments with noisy data (IDP)}

The iniT and the iT method are compared. The setup of our experiments is as follows:
\medskip

--- Exact data $y = A x^\star = PSF * x^\star$ is computed.
    \smallskip

--- Noise of $0.1\%$ and $1\%$ is added to $y$, generating the noisy data
    $y^\delta$. \smallskip

--- The constant function $x_0 = 0$ is used as initial guess for the iT and iniT
    methods. \smallskip

--- The discrepancy principle, with $\tau=1.1$, is used as stopping criterion for
    all methods. \smallskip

--- $\la_k = (\frac23)^k$ in both methods. \smallskip

--- Choose $\alpha_k^\delta$ as in Assumption~\ref{ass:alpha-noise} for the iniT method
    with \smallskip
    
    \centerline{ $\theta_k = (1/k)^{1.1}$. } \smallskip

--- In both iniT and iT methods the computation of $x_{k+1}^\delta$ is performed
    explicitly (in the

    \ \ \ \ frequency domain the convolution corresponds to a multiplication).
    \medskip

\noindent {\bf Numerical test~5 (high/low noise):}
The iniT method and the iT method are implemented for solving the IDP under the
above described setup. Two distinct levels of noise are considered, namely
$\delta = 0.1\%$ and $\delta = 1.0\%$.

In the $\delta = 0.1\%$ noise scenario, the stopping criteria are reached after
25 and 33 steps respectively.
For $\delta = 1.0\%$, the stopping criteria are reached after 6 and 8 steps
respectively.
The progress of the corresponding relative iteration error and relative
residual for both methods are presented in Figure~\ref{fig:ID-evol}.
It is worth noticing that the iniT method applied to this example displays
similar computational savings as discussed in Section~\ref{ssec:num-ipp}.

% ------------------------------------------------------------------------
\section{Final remarks and conclusions} \label{sec:conclusions}

In this manuscript we propose and analyze an implicit two-point type iteration,
namely the {\em inertial iterated Tikhonov} (iniT) method, as an alternative
for obtaining stable approximate solutions to linear ill-posed operator
equations.

The main results discussed in this notes are: boundedness of the sequences
$(x_k)$ and $(w_k)$ generated by the iniT method (Proposition~\ref{prop:aux},
convergence for exact data (Theorem~\ref{th:conv}), stability and
semi-convergence for noisy data (Theorems~\ref{th:stabil}
and~\ref{th:semi-conv} respectively).

We faced an unexpected challenge in our analysis, namely the derivation of a
monotonicity result for the iteration error (i.e. $\norm{x_{k+1} - x^\star}
\leq \norm{x_k - x^\star}$, where $x^\star$ is a solution of $Ax=y$). This
is a standard result in the analysis of iterative regularization methods
\cite{EngHanNeu96, KalNeuSch08}, and is used to establish boundedness and
convergence of $(x_k)$.
Due to the structure of the inertial iteration \eqref{def:iniT}, we are
only able to prove (under the current assumptions) that
$\norm{x_{k+1} - x^\star} \leq \norm{w_k - x^\star}$ (see
Proposition~\ref{prop:aux}~(a)). Notice that, in order to establish the
boundedness of $(x_k)$ in Proposition~\ref{prop:aux}~(b), we had to
resort to Lemma~\ref{lem:appA}, which follows from a result by
Alvarez and Attouch \cite[Theorem~2.1]{AA01}.

The proof of the convergence result in Theorem~\ref{th:conv} uses a novel
strategy. The classical proof is based on a telescopic-sum argument
coupled with the above mentioned monotonicity inequality.
Since the monotonicity of $\norm{x_{k+1}-x^\star}$ is not available, we
used an additional assumption on $(\alpha_k)$ (see Assumption~(A2) in
Theorem~\ref{th:conv}) in order to apply an alternative telescopic-sum
argument (see \eqref{eq:txlj}).

The lack of a monotonicity result also influences the verification of
the stability and semi-convergence results (Theorems~\ref{th:stabil}
and~\ref{th:semi-conv} respectively). The proofs presented here rely
on properties of the sequences $( \alpha_k^{\delta^j} )_k$ and
$( \alpha_k^{\delta^j} \norm{x_k^{\delta^j} - x_{k-1}^{\delta^j}}^2 )_k$.

The choice of $\theta_k$ in Assumptions~\ref{ass:alpha}
and~\ref{ass:alpha-noise} plays a key role in the numerical performance
of the iniT method. In our experiments, we tried $\theta_k = (1/k)^p$ for
distinct choices of $p>1$. The best results were obtained for values of
$p$ close to one.

Our numerical results demonstrate that the iniT method outperforms the
standard iT method (for the same choice of Lagrange multipliers $\la_k$):
\\
$\bullet$ IPP: iniT requires $10\%$ to $20\%$ less CG-steps than iT to reach
the stopping criterion.
\\
$\bullet$ IDP: iniT requires approximately $25\%$ less iterations than iT to
reach the stopping criterion.

{\color{black}
Our numerical experiments cover two of the most relevant families of inverse
problems, namely 'PDE models' and 'integral operators models' \cite{Bau87,
EngHanNeu96, Gr84, Kir96, Rie03}. The benefits
of the proposed iniT method, as compared to the iT method, are readily evident
in all the experiments discussed.
Our numerical results indicate that iniT is a competitive method for solving
other highly ill-posed linear problems within these two families of problems.}

\color{black}
% ------------------------------------------------------------------------
\appendix \setcounter{section}{1}
\section*{Appendix~A} \label{sec:appA}

In what follows we address a result, which is needed for the proof of
Proposition~\ref{prop:aux}. This result corresponds to a (small) part
of the proof of \cite[Theorem~2.1]{AA01}. For the convenience of the
reader, we present here a sketch of the proof.

\begin{lemma} \label{lem:appA}
Let $(\alpha_k)_{k\geq0} \in [0,\alpha]$, with $\alpha \in (0,1)$. Moreover,
let $(\varphi_k)_{k\geq-1}$, $(\eta_k)_{k\geq0}$ be sequences of non-negative
real numbers s.t.
$\varphi_{k+1} - \varphi_k < \alpha_k (\varphi_k - \varphi_{k-1}) + 2\eta_k$,
for $\ k \geq 0$,
\ and \
$\sum_{k\geq0} \eta_k < \infty$.
There exists $\bar\varphi \in \R$ such that $\D\lim_{k\to\infty}
\varphi_k = \bar\varphi$.
\end{lemma}
\begin{proof}
Define $\gamma_k := \varphi_k - \varphi_{k-1}$, for $k \geq 0$. Thus,
it follows from the assumptions
$$
\gamma_{k+1} < \alpha_k \gamma_k + 2\eta_k \ \leq \
\alpha \gamma_k + 2\eta_k \ \leq \
\alpha [\gamma_k]_+ + 2\eta_k \, , \ k \geq 0 \, ,
$$
where $[t]_+ := \max\{t,0\}$, $t \in \R$. Consequently, \
$[\gamma_{k+1}]_+ \leq \alpha [\gamma_k]_+ + 2\eta_k \, , \ k \geq 0$.
A recursive use of this inequality yields
$[\gamma_{k+1}]_+ \leq \alpha^{k+1} [\gamma_0]_+ + 2 \sum_{j=0}^k
\alpha^j \eta_{k-j}$, for $k \geq 0$. Therefore,
$$
\summ_{k=0}^\infty [\gamma_{k+1}]_+ \ \leq \ \D\frac{\alpha}{1-\alpha} [\gamma_0]_+
\ + \ \frac{2}{1-\alpha} \summ_{k=0}^\infty \eta_k \, .
$$
Since $\sum_{k\geq0}\eta_k$ is summable, the right hand side is finite.

Define $\zeta_k := \varphi_k - \sum_{j=1}^k[\gamma_j]_+$, for $k \geq 1$.
Notice that $(\zeta_k)$ is bounded from below. Moreover, $(\zeta_k)$ is
monotone non-increasing. Indeed,
$$
\zeta_{k+1} \ = \ \varphi_{k+1} - [\gamma_{k+1}]_+ - \summ_{j=1}^k[\gamma_j]_+
\ \leq \ \varphi_{k+1} - (\varphi_{k+1} - \varphi_k) - \summ_{j=1}^k[\gamma_j]_+
\ = \ \zeta_k \, , \ k \geq 1 \, .
$$
Consequently, $(\zeta_k)$ converges and we conclude
$\lim_k \varphi_k = \sum_{j\geq1}[\gamma_j]_+ + \lim_k \zeta_k =: \bar\varphi$.
\end{proof}

% --------------------------------------------------------------------
\section*{Acknowledgments}

AL acknowledges support from the AvH Foundation.
ALM acknowledges the financial support funding agency FAPERJ,
from the State of Rio de Janeiro.

% --------------------------------------------------------------------
\bibliographystyle{amsplain}
\bibliography{inertIT}

% --------------------------------------------------------------------
% FIGURES
% --------------------------------------------------------------------

% PAGE 1
%--------------------------------------
\begin{figure}[t]
%
% \centerline{\includegraphics[width=0.6\textwidth]{fig_ExactSol} \hskip-3.3cm
%             \includegraphics[width=0.6\textwidth]{fig_d01_x17}  \hskip-3.3cm
%             \includegraphics[width=0.6\textwidth]{fig_d01_err17}
%
\vskip-8cm  \includegraphics[width=1.0\textwidth]{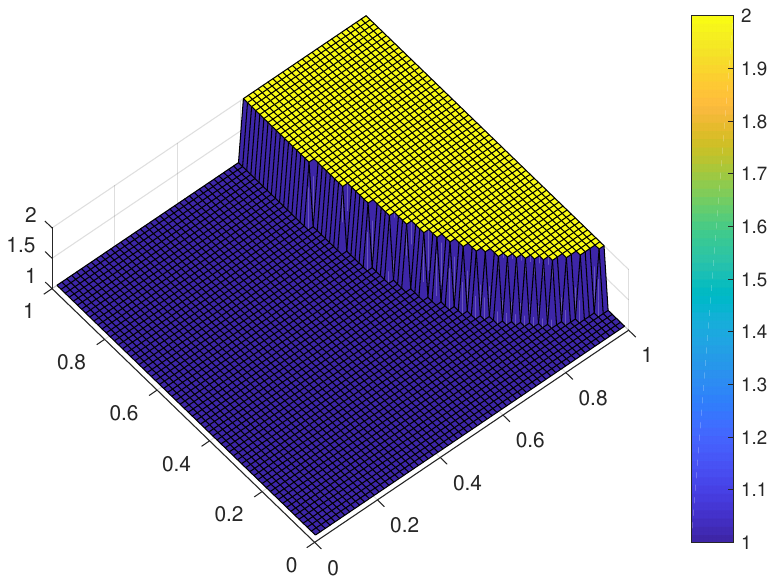}
\vskip-12.5cm \includegraphics[width=1.0\textwidth]{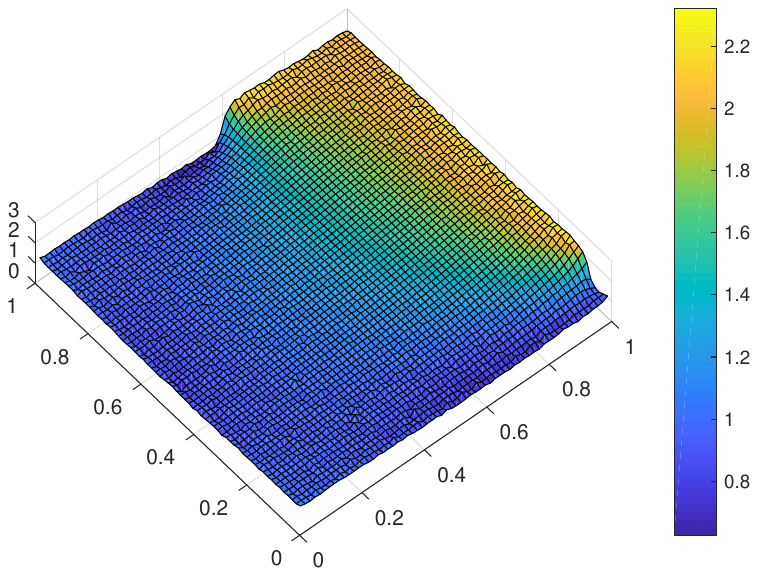}
\vskip-12.5cm \includegraphics[width=1.0\textwidth]{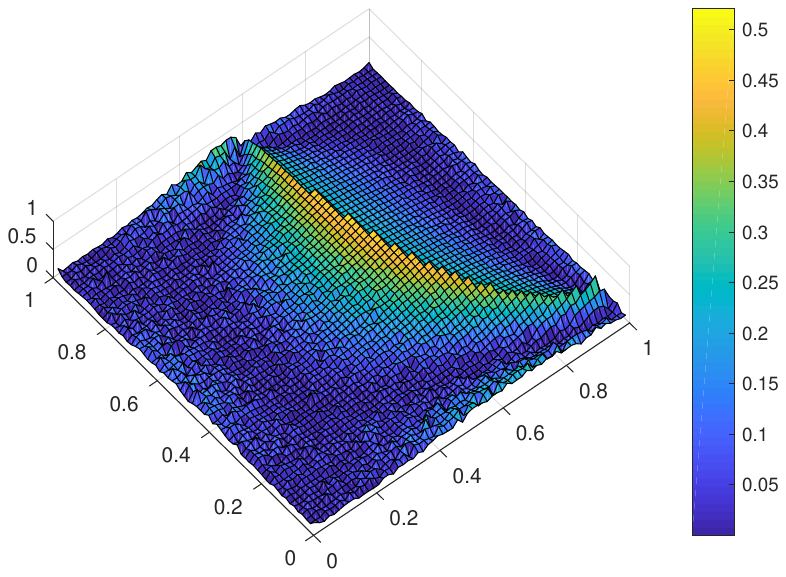}
\vskip-6.5cm
\caption{\small IPP Noise level $0.1\%$. Results obtained by the iniT method (the
stopping criterion is reached at $k^*(\delta) = 17$ steps). (TOP) Ground truth
$x^\star$; (CENTER) Approximate solution $x_{17}^\delta$; (BOTTOM) Relative
iteration error $|x_{17}^\delta-x^\star| / |x^\star|$.}
\label{fig:IPP-setup}
\end{figure}
%--------------------------------------

% PAGE 2
%--------------------------------------
\begin{figure}[t]
\vskip-3.4cm
\centerline{\includegraphics[width=0.75\textwidth]{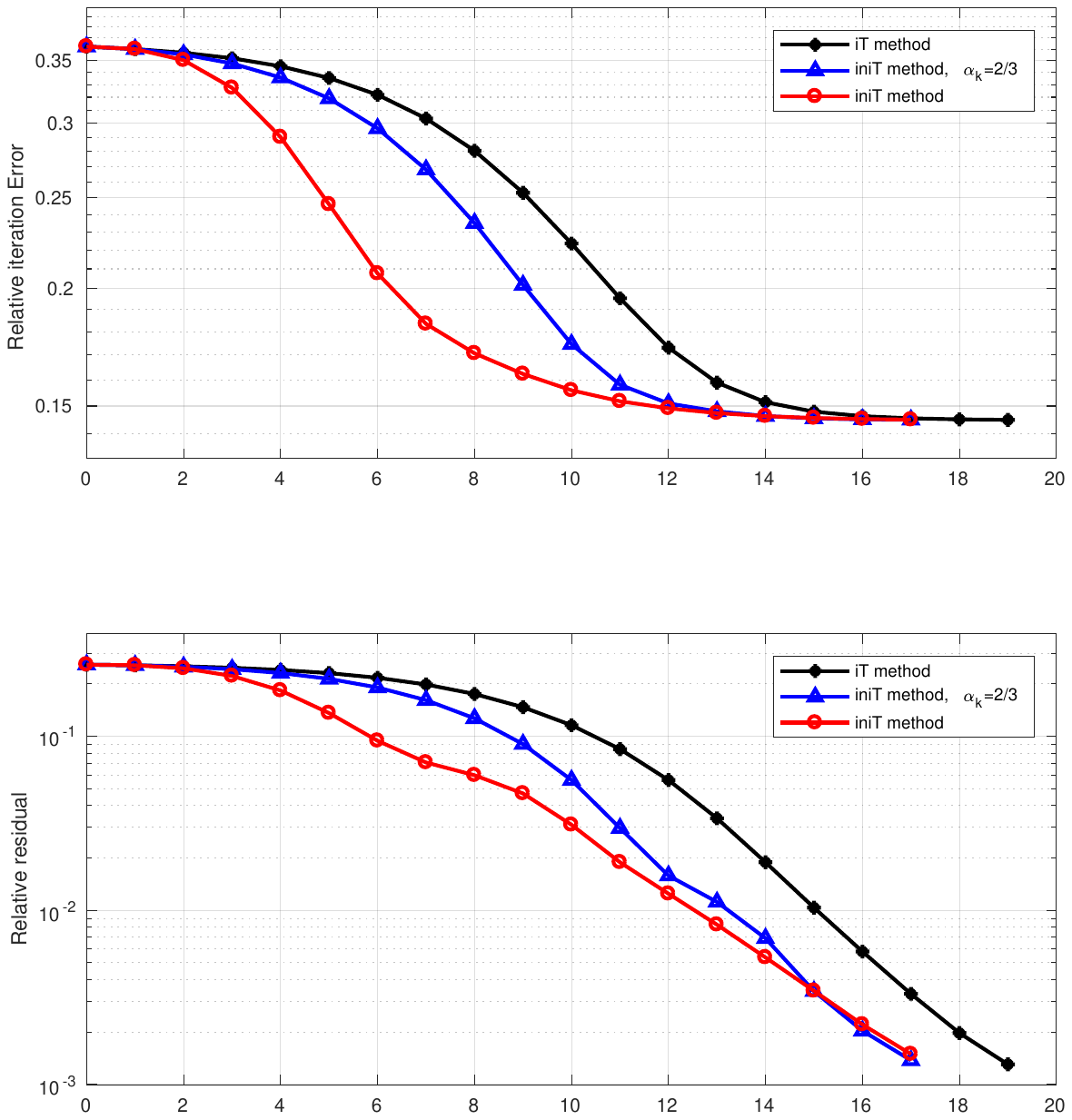}}
\vskip-2.5cm
\caption{\small IPP Noise level $0.1\%$. Comparison between iT and iniT methods.
(TOP) Relative iteration error; (BOTTOM) Relative residual.}
\label{fig:IPP-evolution1}
\end{figure}
%--------------------------------------
%--------------------------------------
\begin{figure}[t]
\vskip-1.3cm
\centerline{\includegraphics[width=0.75\textwidth]{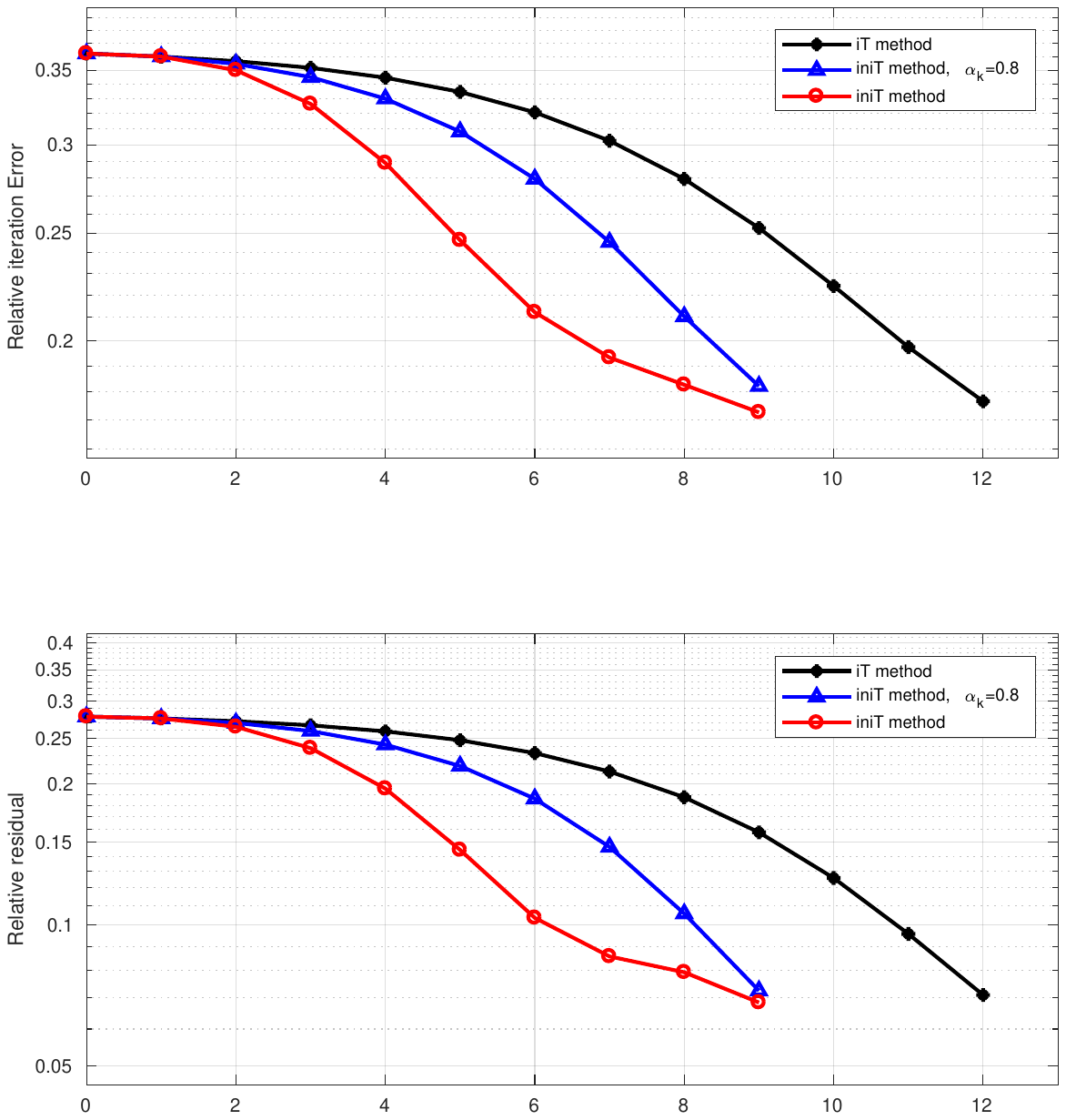}}
\vskip-2.5cm
\caption{\small IPP Noise level $5.0\%$. Comparison between iT and iniT methods.
(TOP) Relative iteration error; (BOTTOM) Relative residual.}
\label{fig:IPP-evolution2}
\end{figure}
%--------------------------------------

% PAGE 3
%--------------------------------------
\begin{figure}[t]
\vskip-3.4cm
\centerline{\includegraphics[width=0.75\textwidth]{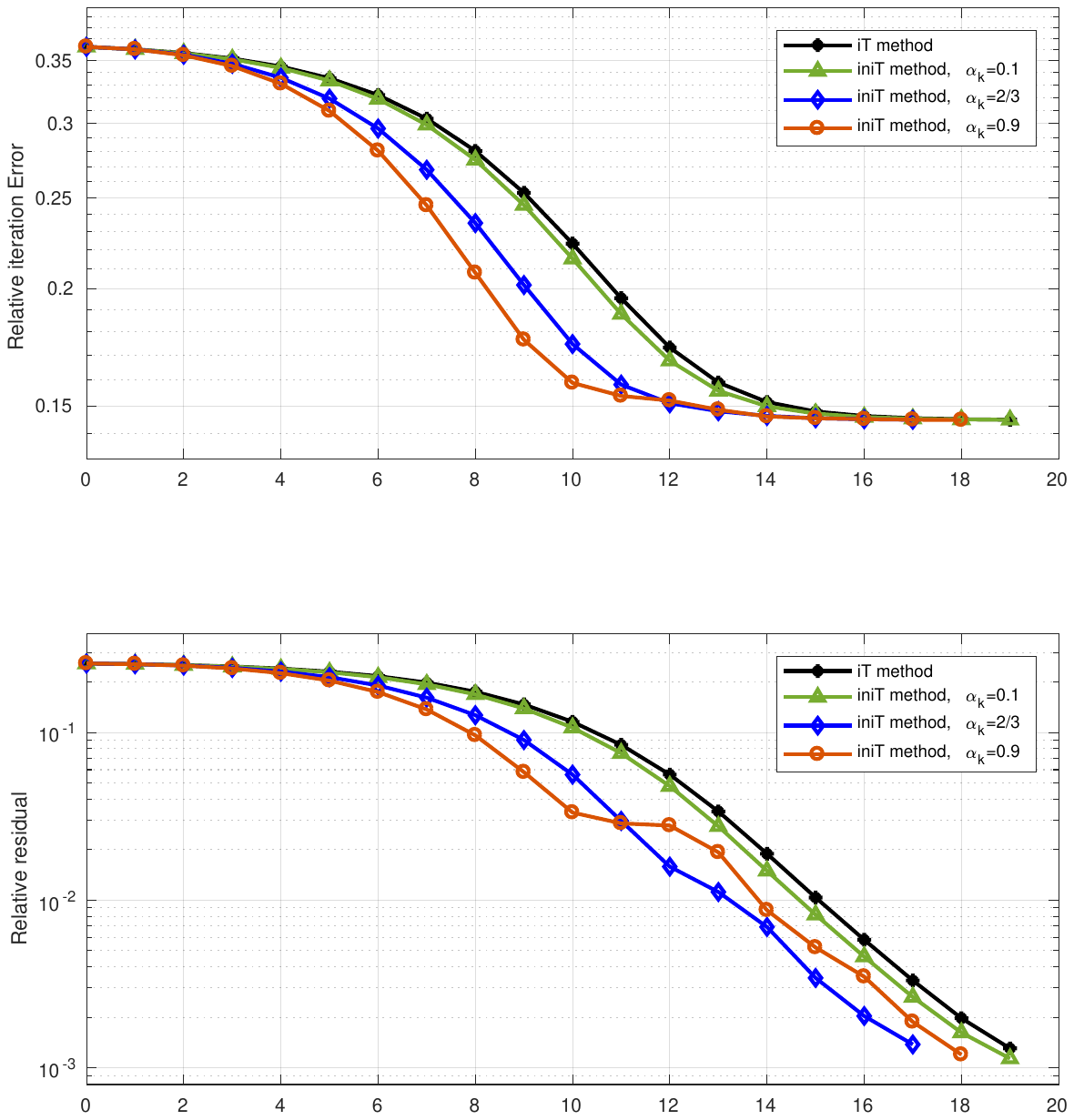}}
\vskip-2.5cm
\caption{\small IPP Noise level $0.1\%$. Comparison between iT method and iniT
method with constant $\alpha_k$. (TOP) Relative iteration error; (BOTTOM)
Relative residual.}
\label{fig:IPP-AlphaConst1}
\end{figure}
%--------------------------------------
%--------------------------------------
\begin{figure}[t]
\vskip-1.3cm
\centerline{\includegraphics[width=0.75\textwidth]{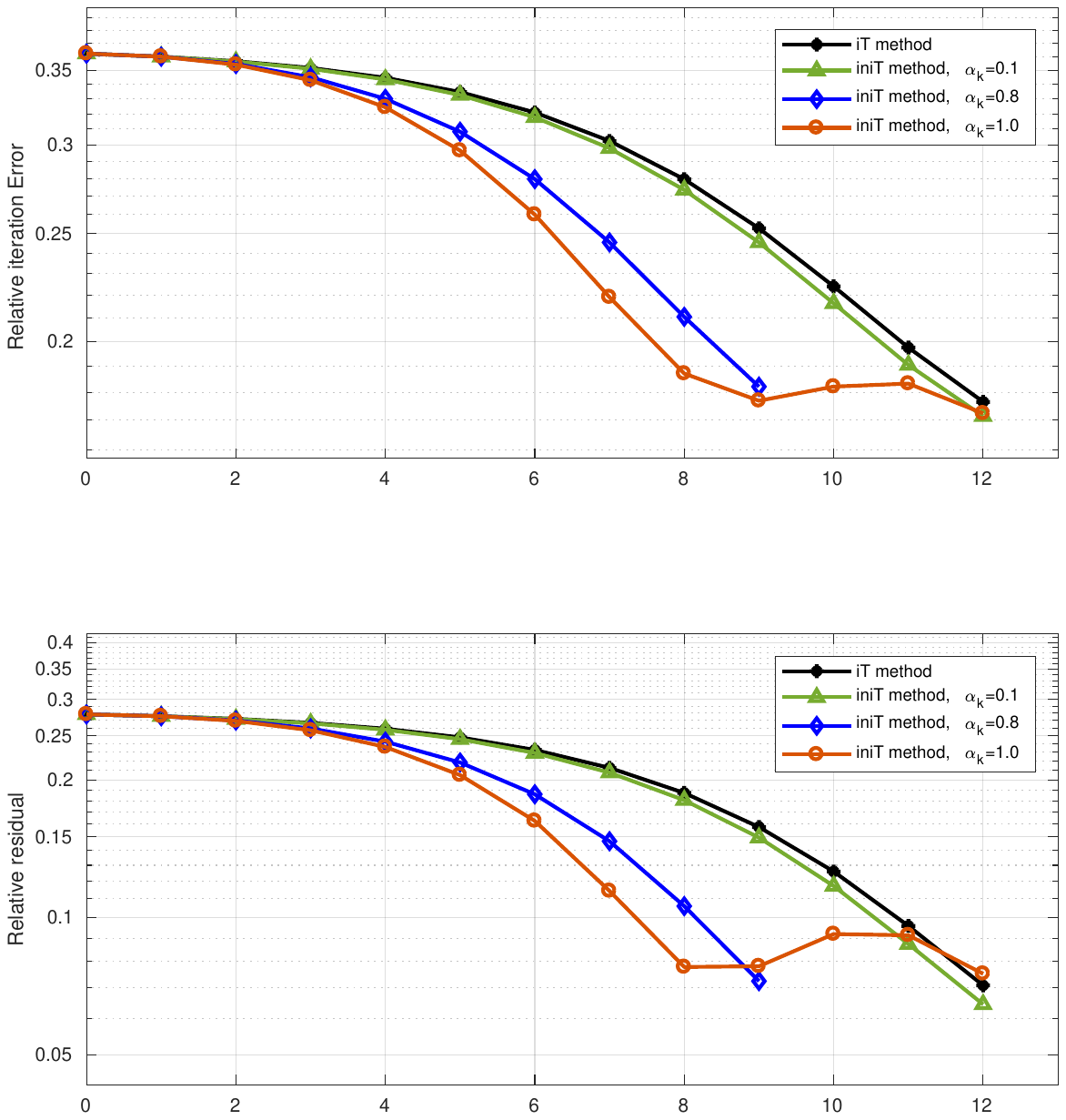}}
\vskip-2.5cm
\caption{\small IPP Noise level $5.0\%$. Comparison between iT method and iniT
method with constant $\alpha_k$. (TOP) Relative iteration error; (BOTTOM)
Relative residual.}
\label{fig:IPP-AlphaConst2}
\end{figure}
%--------------------------------------

\newpage

% PAGE 4
%--------------------------------------
\begin{figure}[t]
\vskip-8cm
\centerline{\includegraphics[width=1.05\textwidth]{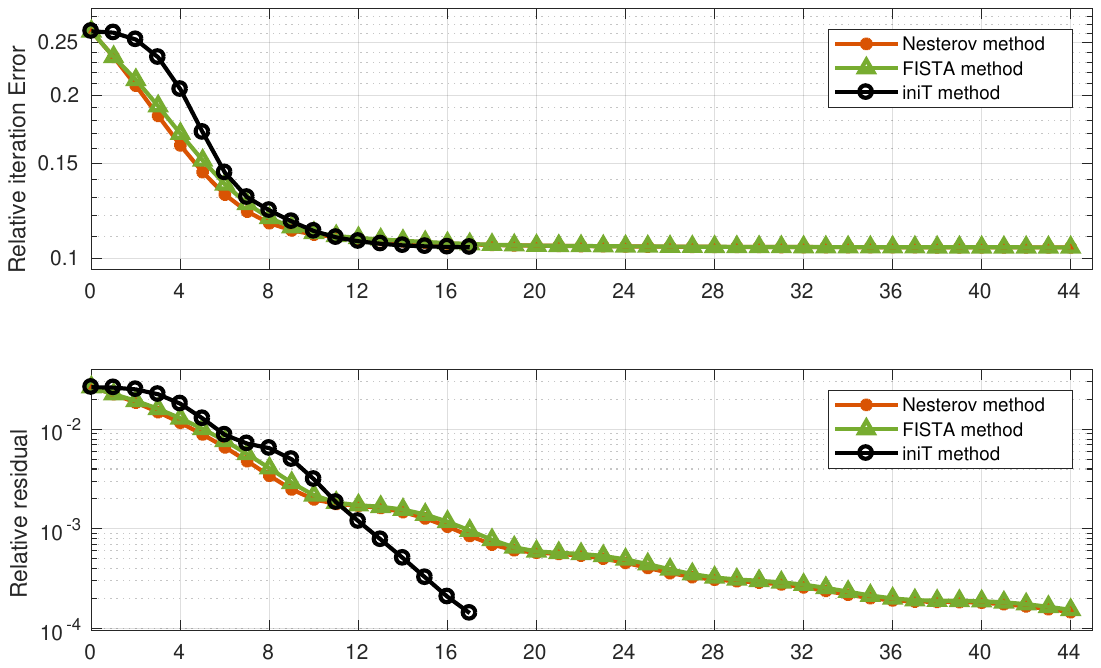}}
\vskip-6.5cm
\caption{\small IPP Noise level $0.1\%$. Comparison between iniT, Nesterov and
FISTA methods. (TOP) Relative iteration error; (BOTTOM) Relative residual.}
\label{fig:IPP-explicit}
\end{figure}
%--------------------------------------
%--------------------------------------
\begin{figure}[t]
\vskip-1.0cm
\centerline{\includegraphics[width=0.4\textwidth]{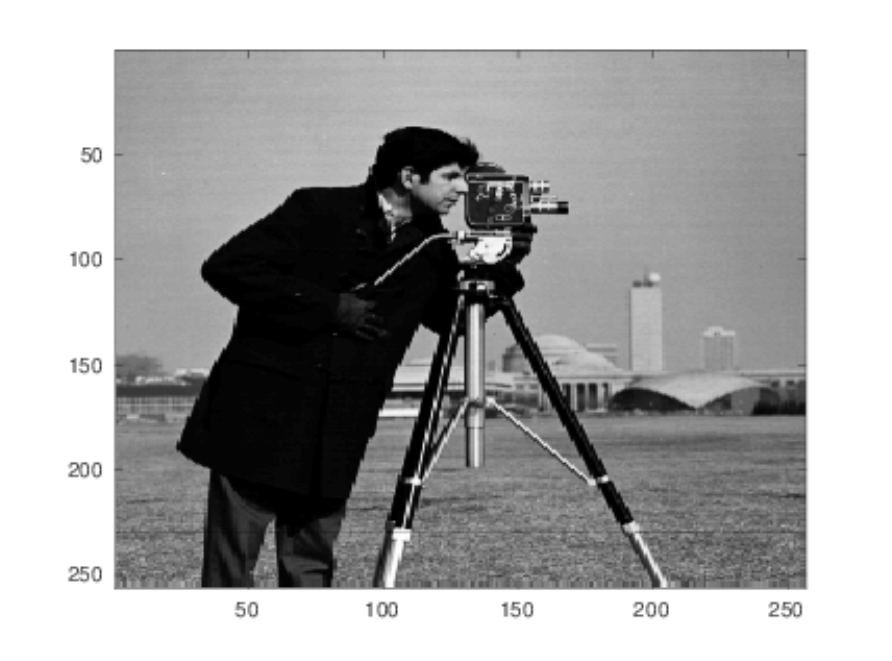} \hskip-0.5cm
            \includegraphics[width=0.42\textwidth]{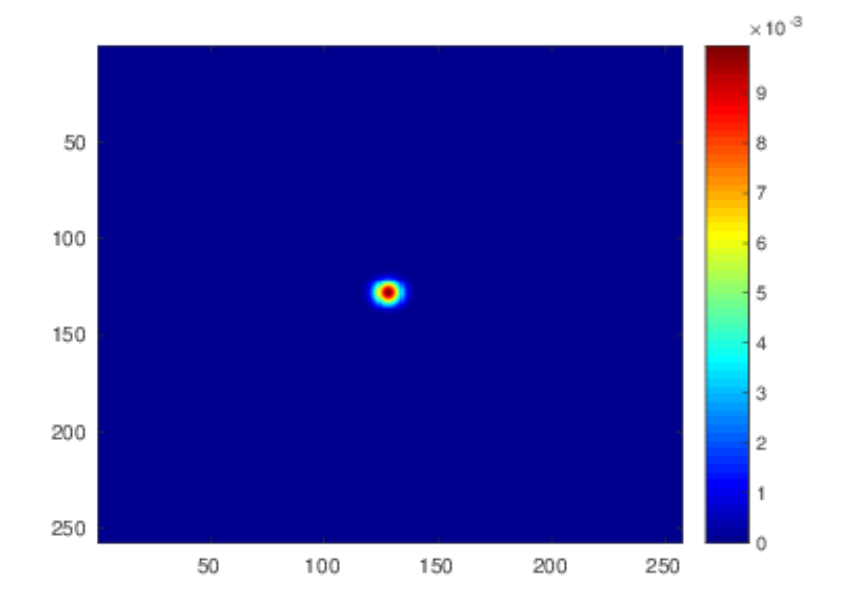}  \hskip-0.8cm
            \includegraphics[width=0.39\textwidth]{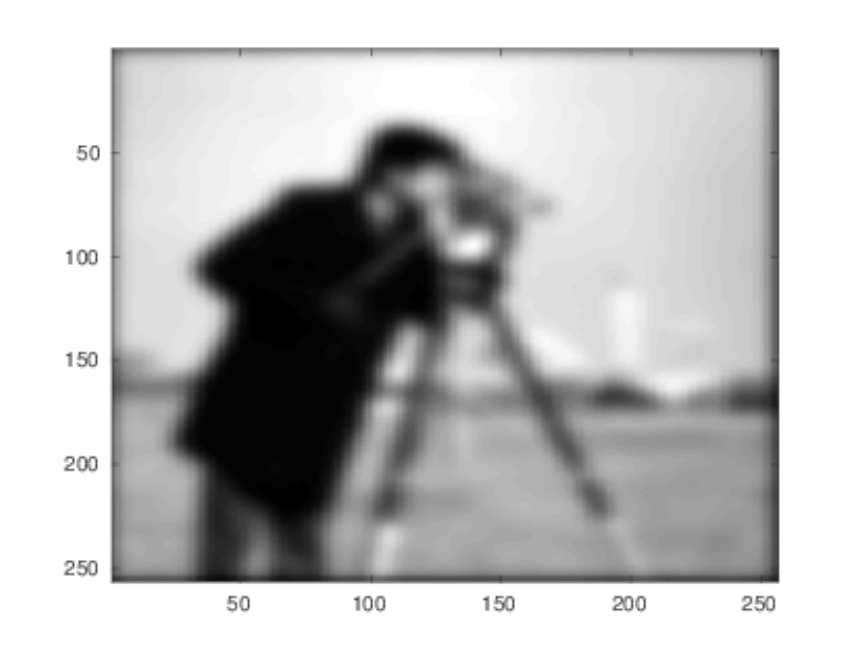}}
\vskip-0.3cm
\caption{\small IDP. Setup: (LEFT) True image;
(CENTER) PSF; (RIGHT) Blurred image.}
\label{fig:ID-setup}
\end{figure}
%--------------------------------------
%--------------------------------------
\begin{figure}[t]
\vskip-4.0cm
\centerline{\includegraphics[width=0.7\textwidth]{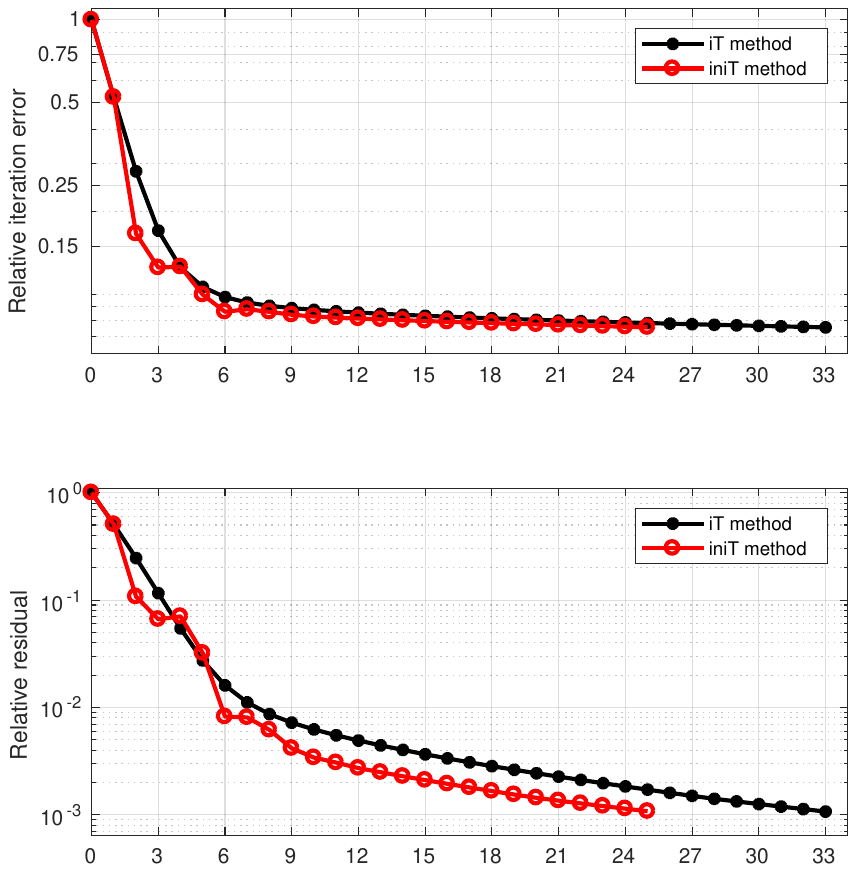} \hskip-3cm
            \includegraphics[width=0.7\textwidth]{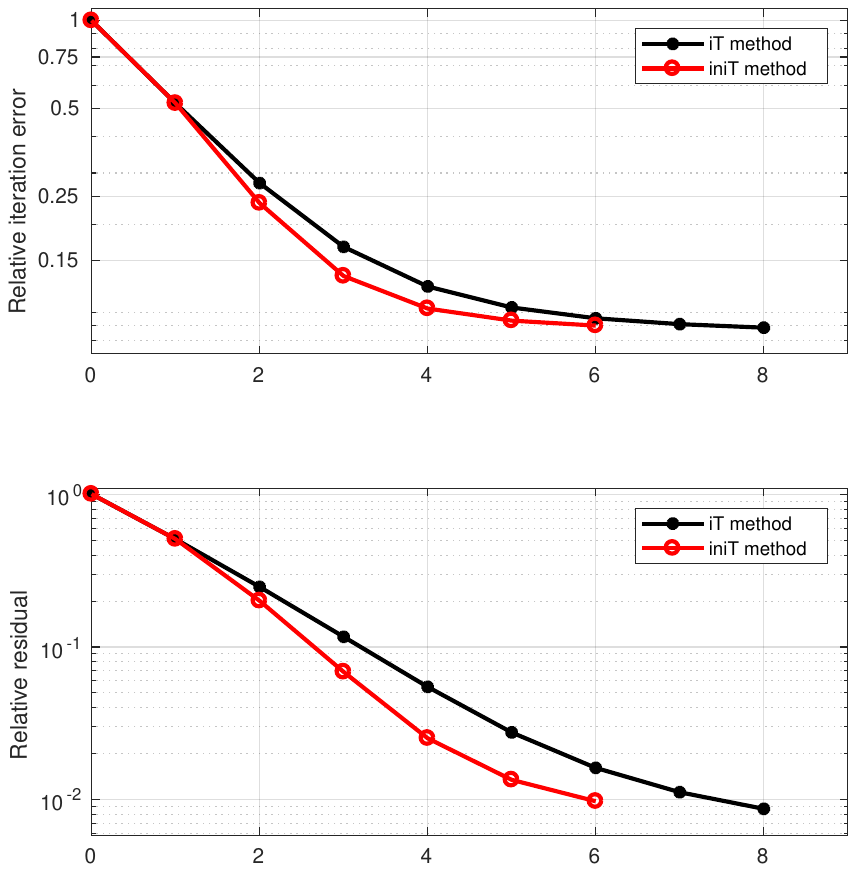}}
\vskip-3.6cm
\caption{\small IDP. Relative iteration error and relative residual: (LEFT) $\delta =
0.1\%$; (RIGHT) $\delta = 1.0\%$.}
\label{fig:ID-evol}
\end{figure}
%--------------------------------------

\end{document}